\documentclass[10pt,a4paper,reqno]{article}
\usepackage[utf8]{inputenc} 
\usepackage[T1]{fontenc} 

\usepackage{libertine}
\usepackage[margin=2.75cm]{geometry}
 
\usepackage{dsfont}
 
\usepackage{amsmath,amsfonts,amssymb,amsthm, stmaryrd,bbm,graphicx,mathtools,enumerate}
  
\usepackage{mathrsfs}
\usepackage[english]{babel}
\usepackage[tracking,spacing=false,kerning,babel]{microtype}
 
\usepackage{framed} 
\usepackage{wrapfig} 

\usepackage{mathalfa}
\usepackage{upgreek}
\usepackage{multicol}
\usepackage{moresize}
\usepackage{graphicx}

\usepackage[final]{hyperref}   
\usepackage[dvipsnames]{xcolor}
\hypersetup{
	colorlinks=true,
	pdfpagemode=UseNone,
    citecolor=Green,
    linkcolor=OrangeRed,
    urlcolor=black,
	pdfstartview=FitW
}
\usepackage{tikz}
\usetikzlibrary{arrows.meta, angles, quotes}
 
\newtheoremstyle{mine}
{\baselineskip}
{\baselineskip}
{\itshape}
{
}
{\bfseries}
{.}
{.5em}
{#1 #2\ifx#3\relax\else~(#3)\fi}

\theoremstyle{mine}

\newtheorem{thm}{Theorem}[section]
\newtheorem{prop}[thm]{Proposition}
\newtheorem{lem}[thm]{Lemma}

\newtheorem{cor}[thm]{Corollary}
\newtheorem{conj}[thm]{Conjecture}

\theoremstyle{remark}
\newtheorem{rem}[thm]{Remark}
        
\colorlet{shadecolor}{blue!10}

\renewcommand{\epsilon}{\varepsilon}

\def\R{{\mathbb R}}
\def\E{{\mathbb E}}
\def\P{{\mathbb P}}

\def\N{{\mathbb N}}

\def\nf{\mathcal{F}}
\def\calN{\mathcal{N}}

\newcommand{\ind}[1]{\mathbbm{1}_{\left\{ #1 \right\}}}
\renewcommand{\epsilon}{\varepsilon}

\newcommand{\cC}{\mathcal{C}}
\newcommand{\cG}{\mathcal{G}}

\newcommand{\cP}{\mathcal{P}}
\newcommand{\cS}{\mathcal{S}}
\newcommand{\cW}{\mathcal{W}}

\DeclareMathOperator{\re}{Re}
\DeclareMathOperator{\im}{Im}
\newcommand{\1}{\mathbbm{1}}
\newcommand{\diff}{\mathop{}\mathopen{}\mathrm{d}}
\newcommand{\abs}[1]{\left\lvert#1\right\rvert}

\let\originalleft\left
\let\originalright\right
\renewcommand{\left}{\mathopen{}\mathclose\bgroup\originalleft}
\renewcommand{\right}{\aftergroup\egroup\originalright}

\newlength{\bibitemsep}
\newlength{\bibparskip}
\let\oldthebibliography\thebibliography
\renewcommand\thebibliography[1]{\oldthebibliography{#1}
	\setlength{\parskip}{\bibitemsep}
	\setlength{\itemsep}{\bibparskip}}

\newcommand{\C}{\mathbb{C}}

\numberwithin{equation}{section}

\title{Fluctuations of additive martingale limits of branching Brownian motion}
\author{Xinxin Chen%
	\footnote{Beijing Normal University, School of Mathematical Sciences, China. Email: xinxin.chen@bnu.edu.cn} 
	\and Michel Pain%
	\footnote{Institut de Mathématiques de Toulouse, Université de Toulouse, CNRS, France. Email: michel.pain@math.univ-toulouse.fr}}

\begin{document}

\maketitle

\begin{abstract}
	Consider a one-dimensional branching Brownian motion. Let $W_\infty(\beta)$ denote the limit of the additive martingale in the subcritical regime $\abs{\beta} < \beta_c$ and $Z_\infty$ be the limit of the derivative martingale at criticality. 
	Madaule (\emph{Stochastic Process. Appl.} 126 (2016), no. 2, 470--502) established the following convergence
	\[
	\frac{W_\infty(\beta)}{\beta_c-\beta}\xrightarrow[\beta\nearrow \beta_c]{\P} 2Z_\infty.
	\]
	The goal of this paper is twofold: firstly, we strengthen this result into an almost sure convergence; secondly, we describe the fluctuations occurring in this convergence by proving
	\[
	\frac{1}{\beta_c-\beta}\left( \frac{W_\infty(\beta)}{\beta_c-\beta} - 2 Z_\infty +2(\beta_c-\beta)\log(\beta_c-\beta) Z_\infty\right)
	\xrightarrow[\beta\nearrow \beta_c]{(\textup{d})} S,
	\]
	where, conditionally on $Z_\infty$, $S$ follows a spectrally negative 1-stable distribution with scale and shift parameters proportional to $Z_\infty$.
	Furthermore, these results are extended to the setting of complex additive martingales and the fluctuations to a multi-dimensional convergence.
\end{abstract}

\medskip 



\section{Introduction}

\subsection{Definitions and motivations}

\paragraph{The model.}

Branching Brownian motion (BBM) is a particle system on the real line defined as follows. 
It starts with one particle at the origin at time $t=0$. Each particle moves according to a Brownian motion with drift $1$ during a lifetime which is exponentially distributed with parameter $\lambda>0$.
When it dies, it is replaced at its position by a random number of children chosen according to the law of some integer-valued random variable $L$.
The displacement, lifetime and number of offspring of a particle are independent, and each particle behaves independently of each other.
We assume that
\begin{equation}\label{hyp}
	\E[L]>1, \quad \E[L^2]<\infty, 
	\quad \text{and} \quad \lambda = \frac{1}{2\E[L-1]}.
\end{equation}
The assumption $\E[L]>1$ ensures that the population survives indefinitely with positive probability.
The particular choice of the drift and of $\lambda$ is made for simplicity of forthcoming formulas, but they are not restrictive because one can modify them up to some affine modification of the process (using the scaling property of Brownian motion).
This choice matches the settings of the papers \cite{MaiPai2019,MaiPai2021}, on which we rely heavily here.

Let $\calN_t$ denote the set of alive particles at time $t$. 
For any $u\in\calN_t$, let $X_u(t)$ denote its position at time $t$ and, for $s\in[0,t]$, write $X_u(s)$ for the position at time $s$ of the unique ancestor of $u$ alive at time $s$. 
Note that in our setting, we have
\begin{equation}
\E\left[\sum_{u\in\calN_t}e^{-X_u(t)}\right]=1,\quad \E\left[\sum_{u\in\calN_t}X_u(t)e^{-X_u(t)}\right]=0,
\quad \text{and} \quad
\E\left[\sum_{u\in\calN_t}X_u(t)^2 e^{-X_u(t)}\right]=t.
\end{equation}
Let $(\nf_t)_{t\ge0}$ denote the natural filtration of the system. 

\paragraph{The additive martingales.}
For $\beta \in \R$, we define
\[
W_t(\beta) \coloneqq \sum_{u\in \calN_t} e^{-\beta X_u(t)-(1-\beta)^2t/2}, \qquad t \geq 0.
\]
Then, $(W_t(\beta))_{t\ge0}$ is a non-negative $(\nf_t)_{t\ge0}$-martingale, which is referred to as the \emph{additive martingale} of parameter $\beta$. 
As a consequence, it has an almost sure limit $W_\infty(\beta)$.
In the context of BBM, these martingales have been introduced by McKean \cite{McK1975} to give a probabilistic construction of traveling wave solutions to the F-KPP equation.
Around the same time, they were also used in the study of the minimal position of a branching random walk (BRW), the discrete-time counterpart of BBM, by Joffe, Lecam and Neveu \cite{JofLeCNev1973} and then Kingman \cite{Kin1975}. See also earlier appearances in \cite{Wat1968}, \cite[Section~VI.4]{AthNey1972} and \cite[Section~7]{JofMon1973}.

The additive martingales exhibit a phase transition, first showed by Biggins \cite{Big1977} for the BRW (see \cite{Nev1988,Kyp2004} for the BBM case): under the optimal assumption $\E[L(\log_+L)]<\infty$,
\begin{itemize}
	\item if $\abs{\beta}<1$, then $(W_t(\beta))_{t\ge0}$ is uniformly integrable and $W_\infty(\beta)>0$ a.s.\@ on the survival event;
	\item if $\abs{\beta} \geq 1$, then $W_\infty(\beta)=0$ a.s.
\end{itemize}
Moreover, these martingales limits play an important role in the understanding of the bulk of the BBM.
Indeed, when $\abs{\beta}<1$, the limit $W_\infty(\beta)$ describes the growth of the BBM around the line of slope $1-\beta$: for example, the number of particles in $(1-\beta)t+[a,b]$ at time $t$, divided by its expectation, converges a.s.\ to $W_\infty(\beta)$, see \cite{Wat1968} for a partial version%
\footnote{Watanabe studies in this paper a general branching Markov process and considers the number of particles in a fixed interval. However, up to adding a drift $-(1-\beta)$ to our definition of the BBM, his result implies the claim stated here as soon as the additive martingale is bounded in $L^2$, that is for $\abs{\beta} < 1/2$.}, 
\cite{Big1979,Uch1982} for the BRW case and \cite[Section~5]{Big1992} for a general result including the BBM case (see also \cite[Theorem~4.7]{Cha2024} for a simple approach written directly for the BBM).

\paragraph{The derivative martingale.}

The critical case $\beta_c=1$ corresponds to the direction of the minimum: indeed, in that case $1-\beta_c=0$ and
\[
	\frac{1}{t} \min_{u\in \calN_t} X_u(t) \xrightarrow[t\to\infty]{(\text{a.s.})} 0.
\]
But $W_\infty(1) = 0$ so the limit of the critical additive martingale is not the right quantity to describe the growth of the BBM in the direction of its minimum.
The appropriate quantity has been found by Lalley and Sellke \cite{LalSel1987} and is defined as follows:
\begin{equation} \label{eq:def_derivative_mart}
	Z_t \coloneqq - \frac{\partial W_t(\beta)}{\partial \beta}\bigg|_{\beta=1}
	= \sum_{u\in\calN_t} X_u(t) e^{-X_u(t)},
	\qquad t \geq 0.
\end{equation}
Then, $(Z_t)_{t\ge0}$ is a $(\nf_t)_{t\ge0}$-martingale, called the \emph{derivative martingale}. It has mean 0, but Lalley and Sellke \cite{LalSel1987} proved that it converges a.s.\@ to a limit $Z_\infty$ which is positive a.s.\@ on the survival event (see \cite{YanRen2011} for a proof under the optimal assumption $\E[L(\log_+L)^2]<\infty$).

There are two other natural ways to try to define a non-trivial limit which would replace $W_\infty(1) = 0$, and it turns out that both methods yield a limit which is a constant multiple of $Z_\infty$.
The first one consists in rescaling the critical additive martingale $W_t(1)$ to obtain a non trivial limit: it has been proved by Aïdékon and Shi \cite{AidShi2014} (see also \cite{BouMai2019} for a simplified approach) that
\begin{equation} \label{eq:seneta-heyde}
	\sqrt{t} W_t(1) 
	= \sqrt{t} \sum_{u\in\calN_t} e^{-X_u(t)}
	\xrightarrow[t\to\infty]{(\P)} 
	\sqrt{\frac{2}{\pi}} Z_\infty.
\end{equation}
The second approach consists in rescaling the subcritical limits $W_\infty(\beta)$ while letting $\beta$ approach $1$ from below: Madaule \cite{Mad2016} proved that 
\begin{equation} \label{eq:left-derivative_of_martingale_limits}
	\frac{W_\infty(\beta)}{1-\beta}
	\xrightarrow[\beta\nearrow 1]{(\P)} 
	2Z_\infty.
\end{equation}
Note that the factor $2$ appearing here, compared to the convergence $Z_t \to Z_\infty$, means that the derivative w.r.t. $\beta$ at $1$ and the limit $t \to \infty$ cannot be interchanged.

Several results have shown that the limit of the derivative martingale $Z_\infty$ is indeed the appropriate quantity to describe the growth in the direction of its minimum.
Initially, it has been used by Lalley and Sellke \cite{LalSel1987} to provide a probabilistic description of the limit of the recentered minimal position, previously obtained by Bramson \cite{Bra1983} in terms of a traveling wave of the F-KPP equation: more precisely, they proved that, for some constant $c_*>0$,  almost surely,
\begin{equation}
\lim_{s\to \infty} \lim_{t\to\infty} 
\P\left( \min_{u\in \calN_t} X_u(t) - \frac{3}{2} \log t \leq x \middle| \nf_s \right) = 1-e^{-c_* Z_\infty e^x},
\end{equation}
which is the cumulative distribution function, conditional on $Z_\infty$, of the opposite of a Gumbel r.v. with shift $\log (c_*Z_\infty)$ and scale 1.
Similarly, $Z_\infty$ also appears in the description of the limit of the extremal process obtained in \cite{AidBerBruShi2013,ArgBovKis2013}.
To study particles at a distance of order $\sqrt{t}$ above the minimum, one can introduce, for any test function $f \colon \R \to \R$, the quantity
\begin{equation} \label{eq:def_Z_t(f)}
Z_t(f) 
\coloneqq  \sum_{u\in\calN_t} X_u(t) e^{-X_u(t)} f \left( \frac{X_u(t)}{\sqrt{t}} \right).
\end{equation}
All these quantities turn out to be asymptotically proportional to $Z_\infty$: 
Madaule \cite{Mad2016} and Maillard and Zeitouni \cite{MaiZei2016} proved that, 
if $f$ is continuous on $(0,\infty)$ and $\abs{f(x)} \leq C(1+x^{-1})$ for any $x>0$, then
\begin{equation} \label{eq:cv_Z_t(f)}
Z_t(f) \xrightarrow[t\to\infty]{(\P)} \rho(f) Z_\infty, 
\qquad \text{where } \rho(f) \coloneqq \int_0^\infty f(x) \sqrt{\frac{2}{\pi}} x^2 e^{-x^2/2} \diff x.
\end{equation}
Note that this convergence implies \eqref{eq:seneta-heyde} by taking $f(x) = 1/x$.
Other quantities describing particles in the direction of the minimum, i.e. with a position $o(t)$ at time $t$, have been considered and their asymptotic behavior is always given by a deterministic multiple of $Z_\infty$, see e.g.\@
\cite{Pai2018} for near-critical additive martingales, \cite[Proposition 2.6]{MaRen2026} for a quantity similar to $Z_t(f)$ but with a different normalization inside the test function, and \cite{Fla2025} for level sets.

\paragraph{Motivations.}

We have seen three natural ways to introduce $Z_\infty$: either as the almost sure limit of the derivative martingale $(Z_t)_{t\geq 0}$, or as the limit in probability of the rescaled critical additive martingale in \eqref{eq:seneta-heyde} and of the rescaled subcritical additive martingale limits in \eqref{eq:left-derivative_of_martingale_limits}. 
A first natural question is whether these last two convergences can be improved to almost sure convergences.
Aïdékon and Shi \cite{AidShi2014} have already proved that the convergence \eqref{eq:seneta-heyde} does not hold almost surely, more precisely they showed 
\[
\limsup_{t \to \infty} \sqrt{t} W_t(1) = \infty, \quad \text{a.s.}
\]
The first main result of this paper (Theorem~\ref{thm:as}) proves that the convergence \eqref{eq:left-derivative_of_martingale_limits} actually holds almost surely. 

A second natural question concerns the fluctuations appearing in these convergences. 
This has already been investigated for the convergence of the derivative martingale in \cite{MaiPai2019} and then for the convergence of the rescaled critical additive martingale \eqref{eq:seneta-heyde} in \cite{MaiPai2021}.
More precisely, the second paper builds on the first one to obtain the fluctuations appearing in the convergence \eqref{eq:cv_Z_t(f)} of $Z_t(f)$ towards $\rho(f) Z_\infty$, for a large class of functions $f$, including $f(x) = 1/x$ for which $Z_t(f) = \sqrt{t} W_t(1)$: the authors prove that
\begin{equation} \label{eq:fluctu_Z_t(f)}
	\sqrt{t} \left( Z_t(f) - \rho(f) Z_\infty - \frac{c(f) \log t}{\sqrt{t}} Z_\infty \right)
	\xrightarrow[t\to\infty]{(\text{d})} S(f),
\end{equation}
where $c(f)$ is a constant depending on $f$ and, conditionally on $Z_\infty$, $S(f)$ has a 1-stable distribution with parameters depending explicitly on $f$.
Our second main result (Theorem~\ref{thm:fluctu}) describes the fluctuations appearing in the convergence \eqref{eq:left-derivative_of_martingale_limits}.

\paragraph{Complex martingales.}

As a tool to prove that the convergence \eqref{eq:left-derivative_of_martingale_limits} holds almost surely, we rely on a contour integral argument involving the additive martingales with complex parameter: for $z \in \C$, let 
\[
W_t(z) \coloneqq \sum_{u\in \calN_t}  e^{-z X_u(t)-(1-z)^2t/2}, \qquad t \geq 0.
\]
Henceforth, we actually obtain our results in the more general setting where the parameter $\beta$ approaching~$1$ from below is replaced by the parameter $z$ approaching $1$ while staying in some region of the complex plane.
We consider the following subsets of the complex plane, which are represented in Figure \ref{fig-1}:
\begin{align}
\cP_1 & \coloneqq 
\left\{ \beta + i \tau \in \C : \abs{\beta} \le \frac12, \beta^2+\tau^2< \frac12 \right\} 
\cup \left\{ \beta+i\tau \in \C : \frac12 < \abs{\beta} < 1, \abs{\tau} < 1 - \abs{\beta} \right\}, \nonumber \\
\cP_{1,2} & \coloneqq 
\left\{ \beta+i\tau \in \C : \frac12 < \abs{\beta} < 1, \abs{\tau} = 1 - \abs{\beta} \right\}. \label{eq:def_phases1}
\end{align}
The open set $\cP_1$ is usually called the \emph{phase~1} of complex additive martingales, and $\cP_{1,2}$ the \emph{boundary 1-2}, as it is at the boundary of both phases 1 and 2, see Section \ref{sec:related_literature} for more details on the phases of complex additive martingales and the related literature.
To state our results, we simply mention here that, by \cite{Big1992,HarKli2015,KolMei2017,HarKli2018}, $\cP_1 \cup \cP_{1,2}$ is exactly the set of points $z \in \C$ such that the martingale $(W_t(z))_{t \geq 0}$ is uniformly integrable%
\footnote{
	The phase $\cP_1$ is covered in \cite{Big1992} for general models including the BBM. 
	The boundary $\cP_{1,2}$ is treated in \cite{KolMei2017} for the BRW, but their result can be applied to the BBM: a BBM along integer times is a BRW so $(W_n(z))_{n\in\N}$ is uniformly integrable and hence closed, which implies $(W_t(z))_{t \geq 0}$ is closed and hence uniformly integrable. 
	The fact that the martingale is not uniformly integrable on $\C \setminus (\cP_1 \cup \cP_{1,2})$ is a consequence of \cite{HarKli2015,HarKli2018}; the case of $\partial \cP_1 \setminus \cP_{1,2}$ is also covered in \cite{KolMei2017} .
}. 
In particular, for any $z\in \cP_1 \cup \cP_{1,2}$ this martingale has a nontrivial limit $W_\infty(z)$.
Moreover, Biggins~\cite[Theorem~6]{Big1992} proved that $(W_t(z))_{t \geq 0}$ converges to $W_\infty(z)$ uniformly on any compact subset of $\cP_1$ almost surely.
As a consequence (see e.g.\ \cite[Theorem 10.28]{Rud1987}), the function $z \mapsto W_\infty(z)$ is analytic on $\cP_1$ a.s.\ and, for any $n \geq 1$, $(\partial^n W_t(z)/(\partial z)^n)_{t \geq 0}$ converges to $\partial^n W_\infty(z)/(\partial z)^n$ uniformly on any compact subset of $\cP_1$ a.s.
\begin{figure}[htbp]
	\centering
	\begin{tikzpicture}[scale=4, >=latex]
	\def\rad{0.707106781}  
	
	\fill[Yellow!50, draw=none]
	(1,0) -- (0.5,0.5) 
	arc[start angle=45, end angle=135, radius=\rad] 
	-- (-0.5,0.5) -- (-1,0) 
	-- (-0.5,-0.5) 
	arc[start angle=-135, end angle=-45, radius=\rad] 
	-- (0.5,-0.5) -- cycle;
	
	\draw[dashed, Black!70] (0,0) -- (0.5,0.5);
	\draw[dashed, Black!70] (0.5,0.5) -- (0.5,-0.5);
	\draw[dashed, Black!70] (-0.5,0.5) -- (-0.5,-0.5);
	\draw[dashed, Black!70] (0,0) circle (\rad);
	\draw[dashed,Black!70,<->] (-115:.02) -- (-115:\rad-.02);
	\node[Black!70] at (-.25,-.3) {$\frac{1}{\sqrt{2}}$};
	
	\draw[->] (-1.2,0) -- (1.25,0) node[below right] {$\beta = \textrm{Re}(z)$};
	\draw[->] (0,-0.8) -- (0,0.85) node[left] {$\tau = \textrm{Im}(z)$};
	
	
	\draw[ForestGreen, very thick] 
	(1,0) -- (0.5,0.5)   
	(0.5,-0.5) -- (1,0)  
	(-1,0) -- (-0.5,0.5) 
	(-0.5,-0.5) -- (-1,0); 
	
	\fill (1,0) circle (0.5pt);
	\fill (-1,0) circle (0.5pt);
	\fill[Black] (0.5,0.5) circle (0.5pt);
	\fill[Black] (0.5,-0.5) circle (0.5pt);
	\fill[Black] (-0.5,0.5) circle (0.5pt);
	\fill[Black] (-0.5,-0.5) circle (0.5pt);
	
	\node at (0.54,0.01) [below] {$\frac12$};
	\node at (0.70710678+.04,0.01) [below] {$\frac{1}{\sqrt{2}}$};
	\node at (1.02,0) [below] {$1$};
	
	\draw[->] (0.2,0) arc[start angle=0, end angle=45, radius=0.2];
	\node at (0.24,0.08) {$\frac{\pi}{4}$};
	
	\node at (-.25,.3) {\Large $\cP_1$};
	\node[ForestGreen] at (.85,.33) {\Large $\cP_{1,2}$};

	\end{tikzpicture}
	\caption{The open set $\cP_1$ is pictured in yellow. The four lines constituting $\cP_{1,2}$ are represented in green, the endpoints are not included.}
	\label{fig-1}
\end{figure}

\subsection{Main results}

\paragraph{Almost sure convergences.}

Our first main result shows that the convergence \eqref{eq:left-derivative_of_martingale_limits} holds almost surely, and extends it to a complex parameter $z$ 
approaching 1 \emph{non-tangentially} in $\cP_1$.

\begin{thm} \label{thm:as}
	Let $\theta \in (0,1)$. 
	Almost surely, the following convergence holds as $\beta \nearrow 1$, uniformly in $\abs{\tau} \leq \theta(1-\beta)$, with $z=\beta+i\tau$,
	\[
	\frac{W_\infty(z)}{1-z} \longrightarrow 2Z_\infty.
	\]
\end{thm}

In particular, Theorem~\ref{thm:as} implies that the function $\beta \mapsto W_\infty(\beta)$ is left-differentiable at $1$, with left-derivative equal to $-2Z_\infty$.
Furthermore, recalling that the derivative $\partial W_\infty(z)/\partial z$ is well-defined on $\cP_1$, another natural question is the convergence of this derivative as $z \to 1$.
While this question would have been non-trivial if we had worked only on the real axis, this convergence is here a direct consequence of Theorem~\ref{thm:as} thanks to the Cauchy integral formula for the derivative of an analytic function and we deduce the following result.

\begin{cor} \label{cor:derivative}
	Let $\theta \in (0,1)$. 
	Almost surely, the following convergence holds as $\beta \nearrow 1$, uniformly in $\abs{\tau} \leq \theta(1-\beta)$, with $z=\beta+i\tau$,
	\[
	\frac{\partial W_\infty(z)}{\partial z} \longrightarrow -2Z_\infty.
	\]
\end{cor}

Our method do not allow to prove the almost sure convergence when $z$ approaches 1 tangentially to the boundary of $\cP_1$, or even along the boundary $\cP_{1,2}$. However, as a byproduct of the proof of Theorem~\ref{thm:as}, we can prove at least the convergence in probability in that case.

\begin{prop}\label{prop:inprob}
	As $z \to 1$ with $z \in \cP_1 \cup \cP_{1,2}$, we have
	\begin{equation*}
	\frac{W_\infty(z)}{1-z}\xrightarrow[]{(\P)} 2Z_\infty.
	\end{equation*}
\end{prop}

Furthermore, we believe the almost sure convergence should still hold and leave it as a conjecture.

\begin{conj} \label{conj:as}
	As $z \to 1$ with $z \in \cP_1 \cup \cP_{1,2}$, we have
	\begin{equation*}
	\frac{W_\infty(z)}{1-z}\xrightarrow[]{(\textup{a.s.})} 2Z_\infty.
	\end{equation*}
\end{conj}

\paragraph{Fluctuations.}

Our second main result describes the fluctuations occuring in this convergence and is stated in terms of $M_{Z_\infty}$, which is defined in \cite[Eq.~(1.7)]{MaiPai2021} as a spectrally positive 1-stable noise (conditionally on $Z_\infty$) with some additional drift. 
Although such a random noise is not a random measure, one can integrate functions against it. 
We refer to \cite{MaiPai2021} for its precise definition in terms of stable integrals, and more generally to \cite[Chapter 3]{SamTaq1994} for the general theory.
We only state here the key properties of $M_{Z_\infty}$ which are sufficient to understand the statement of Theorem~\ref{thm:fluctu} below. 
Define the linear space of functions
\begin{equation} \label{eq:cG}
	\cG \coloneqq \left\{g \colon [0,\infty)\to \R\text{ measurable:} \int_0^\infty \abs{g(r)}(1+\abs{\log\abs{g(r)}}) \frac{\diff r}{r^{3/2}} < \infty \right\},
\end{equation}
where we use throughout the convention $0\log 0 = 0$.
We also introduce the constant $\mu_Z$ defined by
\begin{equation} \label{eq:mu_Z}
\mu_Z \coloneqq \lim_{x\to\infty} \E[Z_\infty\boldsymbol 1_{Z_\infty \le x}] - \log x - \gamma + 1,
\end{equation}
where $\gamma$ is the Euler-Mascheroni constant.
Then, $\int_0^\infty g(r) M_{Z_\infty}(\diff r)$ defines a real valued random variable for any $g \in \cG$, and these random variables satisfy
\begin{align} 
	& \E \left[ \exp \left( i \int_0^\infty g(r) M_{Z_\infty}(\diff r) \right) \middle| Z_\infty \right] \nonumber \\
	& = \exp \left( - \int_0^\infty 
	\left[ \abs{g(r)} + i g(r) \frac{2}{\pi} (\log \abs{g(r)} - \mu_Z) \right] 
	Z_\infty \frac{\sqrt{\pi}}{2\sqrt{2}} \frac{\diff r}{r^{3/2}}\right)
	\label{eq:characteristic_function_M}
\end{align}	
and for any $n \geq 1$, $a_1,\dots,a_n \in \R$ and $g_1,\dots,g_n \in \cG$,
\begin{equation} \label{eq:linearity} 
	\int_0^\infty \left( \sum_{k=1}^n a_k g_k(r) \right) M_{Z_\infty}(\diff r) 
	= \sum_{k=1}^n a_k \int_0^\infty g_k(r) M_{Z_\infty}(\diff r), 
	\qquad \text{almost surely.}
\end{equation}	
Note that these two properties combined uniquely define the joint distribution of $\int_0^\infty g_k(r) M_{Z_\infty}(\diff r)$ for $k = 1,\dots,n$.
Finally, we extend this definition to the case of measurable complex-valued functions $g \colon  [0,\infty)\to \C$ such that $\abs{g} \in \cG$ (or equivalently $\re g \in \cG$ and $\im g \in \cG$), by setting
\[
	\int_0^\infty g(r) M_{Z_\infty}(\diff r)
	\coloneqq \int_0^\infty \re g(r) M_{Z_\infty}(\diff r)
	+ i \int_0^\infty \im g(r) M_{Z_\infty}(\diff r).
\]
We can now state the fluctuation result. It holds for $z$ approaching $1$ along a straight line included in $\cP_1 \cup \cP_{1,2}$ (note that here the boundary $\cP_{1,2}$ is included).
Moreover, the result is stated as a functional convergence. 

\begin{thm} \label{thm:fluctu}
	Let $\cW = \{ w = a+ib \in \C : a > 0, \abs{b} \leq a \}$.
	The following convergence holds in the sense of finite-dimensional distributions, with $z = 1 - \frac{w}{\sqrt{t}}$,
	\begin{equation} \label{eq:fluctu}
	\left( \sqrt{t} \left( \frac{W_\infty(z)}{1-z} - 2 Z_\infty + \frac{w \log t}{\sqrt{t}} Z_\infty \right) \right)_{w \in \cW} 
	\xrightarrow[t \to \infty]{} 
	\left( \int_0^\infty 2 \left( e^{-r w^2/2} -1 \right) M_{Z_\infty}(\diff r) \right)_{w \in \cW}.
	\end{equation}
\end{thm}

\begin{rem} \label{rem:weak_cv_in_proba}
	A stronger statement, similar to the one of \cite[Theorem~1.2]{MaiPai2021}, actually holds: as $t \to \infty$ and then $\varepsilon \searrow 0$, the finite-dimensional conditional distributions of the left-hand side of \eqref{eq:fluctu} given $\mathcal{F}_{\varepsilon t}$ converge weakly in probability towards the finite-dimensional conditional distributions of the right-hand side of \eqref{eq:fluctu} given $Z_\infty$. See \cite[Appendix A]{MaiPai2019} for a brief introduction to weak convergence in probability. 
	To avoid a heavier setting, we restrict ourselves to the weaker statement in Theorem~\ref{thm:fluctu}, but the stronger form follows from the same steps.
\end{rem}

\begin{rem}
	Our approach also identifies the particles responsible for the fluctuations: these are the ones reaching a level $\frac{1}{2} \log t + O(1)$ at a time of order $t$. This means that if the contribution of the descendants of these particles to $W_\infty(z)/(1-z) - 2 Z_\infty$ is removed, then no fluctuations appear at scale $1/\sqrt{t}$. This observation is a consequence of our proof which shows that, up to a negligible error term, $W_\infty(z)/(1-z)$ can be replaced by a quantity of the type $Z_{At}(f)$ with $A$ large but fixed, together with the fact that fluctuations of $Z_{At}(f)$ are due to the particles mentioned above by \cite[Remark~1.6]{MaiPai2021}.
\end{rem}

Restricting ourselves to a one-dimensional convergence and to the case where $z = \beta$ approaches $1$ along the real axis, we deduce the following corollary.
Following~\cite{SamTaq1994}, we define the $1$-stable distribution $\cS_1(\sigma,\beta,\mu)$, with scale parameter $\sigma > 0$, skewness $\beta\in[-1,1]$ and shift parameter $\mu \in \R$, as the distribution on $\R$ having characteristic function 
$\lambda \mapsto \exp( -\sigma (\abs{\lambda} + i \beta \frac{2}{\pi} \lambda \log \abs{\lambda}) + i \mu \lambda)$. 

\begin{cor} \label{cor:fluctu}
	As $\beta \nearrow 1$, 
	\[
	\frac{1}{1-\beta}\left( \frac{W_\infty(\beta)}{1-\beta} - 2 Z_\infty +2(1-\beta)\log(1-\beta) Z_\infty\right)
	\]
	converges in distribution to some non-degenerate limit which is, conditionally on $Z_\infty$, a spectrally negative 1-stable random variable distributed as $\cS_1(\pi Z_\infty, -1, \mu_0 Z_\infty)$ with
	\[
	\mu_0 \coloneqq 2 \Biggl( \log 2 - \mu_Z - \sum_{k\ge 1}\frac{1}{k(\sqrt{k+1}+\sqrt{k})} \Biggr),
	\]
	recalling that $\mu_Z$ is defined in \eqref{eq:mu_Z}.
\end{cor}

\begin{rem}
	It follows that
	\[
		\frac{W_\infty(\beta) -(1-\beta) 2Z_\infty}{(1-\beta)^2} 
		\xrightarrow[\beta \nearrow 1]{(\P)} \infty,
	\]
	which implies that a.s.\ $\beta \mapsto W_\infty(\beta)$ is not twice left-differentiable at $1$.
	In particular, a.s.\ $\partial^2 W_\infty(\beta)/(\partial \beta)^2$ does not converge to a finite limit as $\beta \nearrow 1$.
\end{rem}

\subsection{Related literature and further questions}
\label{sec:related_literature}

\paragraph{Complex additive martingales.}

The idea of studying complex versions of the additive martingales emerged at the same time as the study of the real versions, in order to understand the growth of the bulk of particles in a given direction.
The real additive martingale appears when tilting exponentially the empirical measure so that it is mainly supported in the considered direction, then considering the Fourier transform of this tilted empirical measure makes appear the complex additive martingales. 
This idea was used without tilting in \cite{Wat1968} and then together with the tilting in \cite{Uch1982} for continuous-time BRW and in \cite{Big1992} for discrete-time BRW and more general continuous-time branching processes including the BBM. 
These last two papers establish the $L^p$-convergence of $(W_t(z))_{t\geq0}$ when $z \in \cP_1$ for $p>1$ depending on $z$ (we do not give the definition of $\cP_1$ in this general context here).
See also \cite{IksLiaLiu2019} for necessary or sufficient conditions on $p$ for the $L^p$-convergence.
As mentioned earlier in the introduction, Biggins \cite{Big1992} additionally proved the a.s.\ uniform convergence on compact subsets of $\cP_1$, making use of a contour integral argument which inspired our proof of Theorem~\ref{thm:as}.
This type of uniform convergence had been previously obtained by Joffe, Le Cam and Neveu \cite{JofLeCNev1973} with a very different argument%
\footnote{We find it worth mentioning here: they can compute explicitly the covariances of the limits $W_\infty(\beta)$, then the continuity of the limit follows via Kolmogorov's criterion and finally the uniform convergence can be deduced from the continuity of the limit by a general result on martingales in Banach spaces. Note the difference with \cite{Big1992}, where the regularity of the limit is deduced from the uniform convergence.}
which works in their special case of a binary BRW with Bernoulli($1/2$) jumps, in which the additive martingale converges in $L^2$ for any $\beta \in \R$.

Another reason behind the study of complex additive martingale came from statistical physics. Motivated by directed polymers.
Derrida, Evans and Speer \cite{DerEvaSpe1993} studied the free energy $\lim_{t\to\infty} \frac{1}{t} \log W_t(z)$ for the BRW in a different setting than the one presented so far, which, in the BBM case, corresponds to considering
\begin{equation} \label{eq:general_W_t(z)}
W_t(z) \coloneqq 
\sum_{u\in \calN_t}  e^{-[\beta X_u(t) + i\tau Y_u(t)]-\psi(\beta,\tau) t}, \qquad t \geq 0, 
\quad z = \beta + i \tau \in \C,
\end{equation}
where $(X,Y) = ((X_u(t),Y_u(t)),u\in\calN_t,t\geq 0)$ is a two-dimensional BBM and $\psi(\beta,\tau)$ is chosen so that $\E[W_t(z)] = 1$. 
They proved that, in addition to the second phase already appearing with a real parameter (recall the phase transition at $\beta = 1$), there is a third phase appearing only in the complex plane, thus confirming predictions of \cite{CooDer1990}.
The three phases are $\cP_1$, defined in \eqref{eq:def_phases1},  
\begin{align*}
\cP_2 \coloneqq 
\left\{ \beta + i \tau \in \C : \abs{\beta} > \frac12, \abs{\tau} > 1 - \abs{\beta} \right\} 
\quad \text{and} \quad
\cP_3 \coloneqq 
\left\{ \beta+i\tau \in \C : \abs{\beta} < \frac12, \tau^2 + \beta^2 > \frac{1}{2} \right\}, 
\end{align*}
see Figure~\ref{fig-phases}. Whereas the additive martingale converges to a non trivial limit in $\cP_1$, it converges exponentially fast to 0 in $\cP_2$ and diverges exponentially fast (in an oscillatory manner) in $\cP_3$.
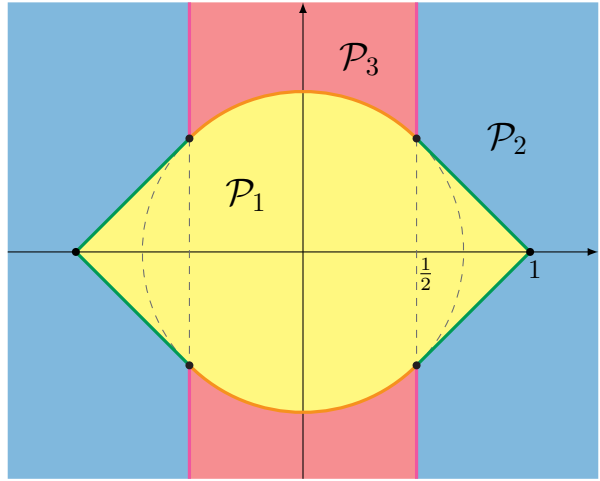
\begin{figure}[htbp]
	\centering
	\begin{tikzpicture}[scale=3, >=latex]
	\def\rad{0.707106781}  
	
	\fill[Yellow, fill opacity=0.5, draw=none]
	(1,0) -- (0.5,0.5) 
	arc[start angle=45, end angle=135, radius=\rad] 
	-- (-0.5,0.5) -- (-1,0) 
	-- (-0.5,-0.5) 
	arc[start angle=-135, end angle=-45, radius=\rad] 
	-- (0.5,-0.5) -- cycle;
	
	\fill[Red, fill opacity=0.5, draw=none]
	(0.5,0.5) 
	arc[start angle=45, end angle=135, radius=\rad] 
	-- (-.5,1.1) -- (.5,1.1) -- cycle;
	\fill[Red, fill opacity=0.5, draw=none]
	(-0.5,-0.5) 
	arc[start angle=-135, end angle=-45, radius=\rad] 
	-- (0.5,-1) -- (-.5,-1) -- cycle;
	
	\fill[RoyalBlue, fill opacity=0.5, draw=none]
	(0.5,0.5) -- (1,0) -- (0.5,-0.5) -- (0.5,-1) -- (1.3,-1) -- (1.3,1.1) -- (.5,1.1) -- cycle; 
	\fill[RoyalBlue, fill opacity=0.5, draw=none]
	(-0.5,0.5) -- (-1,0) -- (-0.5,-0.5) -- (-0.5,-1) -- (-1.3,-1) -- (-1.3,1.1) -- (-.5,1.1) -- cycle; 
	
	\draw[->] (-1.3,0) -- (1.3,0); 
	\draw[->] (0,-1) -- (0,1.1); 
	
	\draw[dashed, Black!70] (-0.5,0.5) arc[start angle=135, end angle=225, radius=\rad];
	\draw[dashed, Black!70] (0.5,-0.5) arc[start angle=-45, end angle=45, radius=\rad];
	\draw[dashed, Black!70] (0.5,0.5) -- (0.5,-0.5);
	\draw[dashed, Black!70] (-0.5,0.5) -- (-0.5,-0.5);
	
	\draw[BurntOrange, very thick] (0.5,0.5) arc[start angle=45, end angle=135, radius=\rad];
	\draw[BurntOrange, very thick] (0.5,-0.5) arc[start angle=-45, end angle=-135, radius=\rad];
	
	\draw[ForestGreen, very thick] 
	(1,0) -- (0.5,0.5)   
	(0.5,-0.5) -- (1,0)  
	(-1,0) -- (-0.5,0.5) 
	(-0.5,-0.5) -- (-1,0); 
	
	\draw[VioletRed, very thick] 
	(.5,1.1) -- (0.5,0.5)   
	(0.5,-0.5) -- (.5,-1)  
	(-.5,1.1) -- (-0.5,0.5) 
	(-0.5,-0.5) -- (-.5,-1); 
	
	\fill (1,0) circle (0.5pt);
	\fill (-1,0) circle (0.5pt);
	\fill[Black] (0.5,0.5) circle (0.5pt);
	\fill[Black] (0.5,-0.5) circle (0.5pt);
	\fill[Black] (-0.5,0.5) circle (0.5pt);
	\fill[Black] (-0.5,-0.5) circle (0.5pt);
	
	\node at (0.54,0.015) [below] {$\frac12$};
	\node at (1.02,0) [below] {$1$};
	
	\node at (-.25,.25) {\Large $\cP_1$};
	\node at (.9,.5) {\Large $\cP_2$};
	\node at (.25,.85) {\Large $\cP_3$};
	\end{tikzpicture}
	\caption{The three phases $\cP_1$, $\cP_2$ and $\cP_3$ are the open sets pictured in yellow, blue and red respectively. The boundaries $\cP_{1,2}$, $\cP_{1,3}$ and $\cP_{2,3}$ are the lines in green, orange and violet respectively. The black dots are excluded from all these sets.}
	\label{fig-phases}
\end{figure}

Complex additive martingales are also an important tool for the construction of complex multiplicative cascades, random complex-valued measures with multifractal properties obtained by exponentiating a BRW \cite{BarJinMan2010a,BarJin2010}.
These cascades can also be considered as a toy model for Gaussian multiplicative chaos, which we will discuss in more detail later.

Recently, more precise results have been obtained on the phases $\cP_2$ and $\cP_3$ as well as on their boundaries: we already introduced $\cP_{1,2}$ in \eqref{eq:def_phases1} and we write $\cP_{1,3}$ for the intersection of the boundaries of $\cP_1$ and $\cP_3$ (without the triple points $z = \pm(1\pm i)/2$) and similary for $\cP_{2,3}$.
Moreover, some of these results are written in the more general context where $W_t(z)$ is defined as in \eqref{eq:general_W_t(z)} 
but the BBMs $X$ and $Y$ have a correlation proportional to $\lambda \in [0,1]$, that is $Y_u(t) = \sqrt{\lambda} X_u(t) + \sqrt{1-\lambda} \tilde{Y}_u(t)$, where $(X,\tilde{Y})$ is a two-dimensional BBM.
Note that the case considered in this paper is the fully correlated case where $\lambda=1$ and that the independent case corresponds to $\lambda=0$.
Recall that, for $z \in \cP_1 \cup \cP_{1,2}$, $W_t(z)$ converges a.s. without renormalization to a non-trivial limit in the $\lambda = 1$ case \cite{Big1992,KolMei2017}.
The phase $\cP_1$ for any $\lambda \in [0,1]$ is covered in \cite{HarKli2018}\footnote{
	The boundary $\cP_{1,2}$ is also dealt with in \cite{HarKli2018} but we believe there is a flaw in their argument, located in the first line of (4.12).
}. 
For $z = \beta+i\tau \in \cP_2$, called the glassy phase, it has been proved in \cite{MadRhoVar2015} for $\lambda=0$ and then in \cite{HarKli2015} for any $\lambda$ that, after renormalization,  
$W_t(z)$ converges in distribution towards a non-trivial limit, whose distribution, conditionally on $Z_\infty$, is complex isotropic $1/\beta$-stable when $\lambda<1$.
For $z = \beta+i\tau \in \cP_3$, it has been showed in \cite{BarJinMan2010a} for the BRW and in \cite{HarKli2018} for the BBM that, after an exponential renormalization,  
$W_t(z)$ converges in distribution to a limit, which is, conditionally on $W_\infty(2\beta)$, normally distributed with variance $W_\infty(2\beta)$.
In \cite{HarKli2018}, the authors also consider the boundaries $\cP_{1,3}$, $\cP_{2,3}$ and the triple point $\{ \pm(1\pm i)/2 \}$.

In this paper, we study the behavior of $W_\infty(z)$ as $z$ approaches 1 in $\cP_1$. It is natural to also consider the behavior as $z$ approaches other parts of the boundary of $\cP_1$.
When $z$ approaches $\cP_{1,2}$, we expect to get the value of $W_\infty(z)$ on the boundary in the limit. This leads to the following conjecture.

\begin{conj} \label{conj:continuity}
	Almost surely, the function $z \mapsto W_\infty(z)$ is continuous on $\cP_1 \cup \cP_{1,2}$.
\end{conj}

For $z = \beta+i\tau \in \cP_{1,3}$, the limit $W_\infty(z)$ is not well-defined: instead the additive martingale $W_t(z)$ explodes as $t\to\infty$ and, after a polynomial renormalization, converges to a r.v.\ whose distribution, conditionally on $W_\infty(2\beta)$, is normal with variance $W_\infty(2\beta)$. 
We expect a similar behavior for $W_\infty(z)$ when $z$ approaches $\cP_{1,3}$, as detailed in the following conjecture. 
This is also supported by \cite[Theorem~3.20]{JunSakWeb2020}, which establishes this result when $z$ is purely imaginary and for Gaussian multiplicative chaos, a related model discussed below.

\begin{conj} \label{conj:CLT_13}
	Let $z_0 = \beta_0 + i \tau_0 \in \cP_{1,3}$. There exists a constant $\sigma_0>0$ such that, as $z \to z_0$ with $z \in \cP_1$,
	\[
		\left( \frac{1}{2} - \abs{z}^2 \right)^{1/2} W_\infty(z)
		\xrightarrow{(\textup{d})} N,
	\]
	where, conditionally on $W_\infty(2\beta_0)$, the random variable $N$ has distribution $\mathcal{N}(0,\sigma^2_0 W_\infty(2\beta_0))$.
	At the triple point $z_0 = (1+i)/2$, the same result holds but with exponent $1/4$ instead of $1/2$ and with $Z_\infty$ instead of $W_\infty(2\beta_0)$.
\end{conj}

\paragraph{Gaussian multiplicative chaos.}

The study of Gaussian multiplicative chaos (GMC) dates back to Kahane \cite{Kah1985} with motivation coming from the theory of intermittent turbulence, but has known a renewed interest due to its key role in the probabilistic construction of Liouville quantum gravity initiated in \cite{DavKupRhoVar2016}, see \cite{BerPow2026,RhoVar2026} for recent reviews on the topic.

Consider a log-correlated Gaussian field $X$ defined on a domain $D \subset \R^d$, for example the Gaussian free field in dimension 2.
The Gaussian multiplicative chaos (GMC) with parameter $\gamma \in \R$ is the random measure on $D$ defined informally as
\[
	\mu^\gamma (\diff x) \coloneqq e^{\gamma X(x) - \frac{\gamma^2}{2} \E[X(x)^2]} \diff x,
\]
which does not directly make sense because $X$ is only defined as a random distribution and $\E[X(x)^2] = \infty$. Instead, one first considers a regularized version of the field $X_\varepsilon$ at some scale $\varepsilon>0$, for example via convolution or truncation of the high modes, and then one defines $\mu^\gamma$ as the weak limit as $\varepsilon \to 0$ of the following measures
\[
\mu^\gamma_\varepsilon (\diff x) \coloneqq e^{\gamma X_\varepsilon (x) - \frac{\gamma^2}{2} \E[X_\varepsilon (x)^2]} \diff x,
\]
where the convergence holds almost surely for some specific regularizations and only in probability in general.
Furthermore, there exists some explicit $\gamma_c > 0$ such that the limiting measure is non-zero iff $\abs{\gamma} < \gamma_c$; for simple proofs see \cite{Ber2017} for the case $\abs{\gamma} < \gamma_c$ and \cite[Lemma 2.1]{Pow2021} for the case $\abs{\gamma} \geq \gamma_c$.
This phase transition is reminiscent of the one for additive martingales. 
Indeed, the BBM and the BRW at time $t$ are toy models for $X_\varepsilon$ with the correspondance $t = \log(1/\varepsilon)$, see e.g.\ \cite{Arg2017}, and so $W_\infty(\beta)$ is the analog of the total mass of $\mu^\gamma$. 
More precisely, multiplicative cascades mentioned earlier are the BRW analog of the GMC, see also \cite{AruPowSep2020} for a way to use this link rigorously.

As for the BBM, one can construct a non-trivial measure at the critical value $\gamma=\gamma_c$ by one of the three following way, which yield the same measure up to explicit multiplicative constants: 
\begin{itemize}
	\item Consider the derivative of $\mu^\gamma_\varepsilon$ w.r.t.\ $\gamma$ at $\gamma_c$ and then take the weak limit, a.s. or in probability depending on the regularization, as $\varepsilon \to 0$, in a way similar to the convergence of the derivative martingale, see \cite{DupRhoSheVar2014a};
	\item Rescale $\mu^\gamma_\varepsilon$ by a factor $\sqrt{\log(1/\varepsilon)}$ and then take the weak limit in probability as $\varepsilon \to 0$, as in \eqref{eq:seneta-heyde} for the BBM, see \cite{DupRhoSheVar2014b};
	\item Take the weak limit in probability of $\mu^\gamma/(\gamma_c-\gamma)$ as $\gamma \nearrow \gamma_c$, as in \eqref{eq:left-derivative_of_martingale_limits} for the BBM, see \cite{AruPowSep2019}.
\end{itemize}
These results were first proved for specific log-correlated fields, but then extended to more general ones by \cite{JunSakWeb2019}.
See \cite{Pow2021} for a review and \cite{Lac2024a} for simpler proofs.
Our Theorem~\ref{thm:as} suggests that the convergence in the third point could be reinforced into an almost sure convergence.

\begin{conj} \label{conj:GMC}
	For Gaussian log-correlated fields with sufficiently regular covariance function, $\mu^\gamma/(\gamma_c-\gamma)$ converges almost surely as $\gamma \nearrow \gamma_c$.
\end{conj}

Similarly, one can construct complex GMC measures by exponentiating $\gamma X + i \tau Y$ with $Y$ a log-correlated field, which is usually chosen either equal to $X$ or independent of $X$; these two options correspond respectively to the cases $\lambda= 1$ and $\lambda=0$ mentioned above for the complex additive martingale of the BBM.
The exact same phase diagram appears, up to a scaling such that $\beta_c=1$ corresponds to $\gamma_c$.
In the case $X=Y$, the phase $\cP_1$ has been treated in \cite{BarJinMan2010b} and \cite[Appendix A]{AstJonKupSak2011} for specific fields and in \cite[Theorem~6.1]{JunSakWeb2019} for general fields. 
Furthermore, it is proved in these papers that the complex GMC is analytic in its parameter $z = \gamma + i\tau \in \cP_1$.
We also mention \cite{Jeg2026} which proves that the GMC for $\gamma \in (-\gamma_c,\gamma_c)$ can be extended analytically to $\cP_1$ solely based on the study of the derivatives w.r.t.\ $\gamma$ on the real line.
In \cite{LacRhoVar2015}, phases $\cP_1$ and $\cP_3$ as well as boundaries $\cP_{1,2}$, $\cP_{1,3}$ and $\cP_{2,3}$ (without the triple points) have been treated when $X$ and $Y$ are independent for some specific fields. 
In a series of papers \cite{Lac2022a,Lac2022b,Lac2024b}, Lacoin extended these results to more general fields, included the case $X=Y$, and covered the triple points as well.
Thus, only phase $\cP_2$ remains open.
The behavior of complex GMC measures as the parameter $z \in \cP_1$ approaches the boundary of $\cP_1$ is also open and should be given by statements similar to Conjectures~\ref{conj:as}, \ref{conj:continuity} and \ref{conj:CLT_13}. Recall that the case where $z$ is purely imaginary as been treated in \cite{JunSakWeb2020}.

Finally, subcritical derivative GMC measures, whose total mass is analogous to the quantity $\partial W_\infty(\beta)/\partial \beta$ appearing in Corollary \ref{cor:derivative}, also play a role in the probabilistic construction of the quantum Mabuchi theory \cite{LacRhoVar2022}. 
This has raised some recent interested in the study of the left tail of $\partial W_\infty(\beta)/\partial \beta$ for $\abs{\beta} < \beta_c$, see \cite{BonVar2023,CheHuaMa2025}.

\paragraph{Fluctuations.}

As already mentioned, the fluctuation result obtained in Theorem~\ref{thm:fluctu} is closely linked to the 1-stable fluctuations obtained for the derivative martingale in \cite{MaiPai2019} and more generally for $Z_t(f)$ in \cite{MaiPai2021}, recall \eqref{eq:fluctu_Z_t(f)}, including in particular the rescaled critical additive martingale.
More precisely, the proof of Theorem~\ref{thm:fluctu} directly relies on results from \cite{MaiPai2021}.
Fluctuations of the derivative martingale have been generalized to the BRW in \cite{BurIksMal2021} for the one-dimensional convergence and in  \cite{HouRenSon2024} for the multi-dimensional convergence.
Other 1-stable fluctuations have been obtained for the level sets of the limiting extremal process in \cite{MytRoqRyz2022} via a PDE approach and in \cite{HarLouWu2025} with probabilistic techniques.

The study of fluctuations of additive martingales around their limits dates back to \cite{RosTopVat2002} in the context of weighted branching processes which includes the BRW.
In the context of BBM another phase transition occurs at $\beta_c/2$.
When $\abs{\beta}< \beta_c/2$, after an exponential renormalization, $W_t(\beta)-W_\infty(\beta)$ converges in distribution to a normal r.v.\ with variance $W_\infty(2\beta)$ conditionally on $W_\infty(2\beta)$: this was already observed in \cite{RosTopVat2002}, then generalized to multiplicative cascades in \cite[Theorem 2.8]{BarJinMan2010a}, to a multi-dimensional convergence for the BRW in \cite{IksKab2016}. 
The paper \cite{BarJinMan2010a} also includes the complex setting where $\beta$ is replaced by $z = \beta + i \tau \in \cP_1$ with $\abs{\beta}< \beta_c/2$, see also \cite{HarKli2018} for the BBM case.
Fluctuations in the full region $\cP_1 \cup \cP_{1,2}$ where the limit $W_\infty(\beta)$ is non-trivial were obtained in \cite{IksKolMei2020} for the BRW. When $\abs{\beta} = \beta_c/2$ the limit is still normal, but requires an additional polynomial renormalization and now has variance $Z_\infty$.
When $\abs{\beta} > \beta_c/2$, the limit is either described in terms of the limiting extremal process when $z \in \cP_1$, or as a complex $\alpha$-stable Lévy process evaluated at time $Z_\infty$ when $z \in \cP_{1,2}$, where $\alpha \in (1,2)$ and $\alpha=\beta_c/\beta$ for the BBM.
We believe that the latter description in terms of an $\alpha$-stable Lévy process also holds when $z \in \cP_1$ by properties of the exponential PPP underlying the limiting extremal process: this has been observed in \cite[Theorem 5.1]{Cha2024} for the BBM when $z = \beta \in \R$.
Note that the stability parameter $\alpha$ approaches 1 as $z \to \beta_c$.
Furthermore, when $\beta > \beta_c$, the additive martingale tends to 0 but it can be renormalized to converge to an $\alpha$-stable Lévy process evaluated at time $Z_\infty$, with $\alpha=\beta_c/\beta$ taking now values in $(0,1)$, see \cite{BarRhoVar2012}.

\subsection{Proof strategy and organization of the paper}

Proofs of the two main results, Theorems~\ref{thm:as} and \ref{thm:fluctu}, are based on the same ingredient: a concentration result for $W_\infty(z)/(1-z)$ around a quantity of the type $Z_{At}(f)$ when $z = 1 - \frac{w}{\sqrt{t}} \in \cP_1 \cup \cP_{1,2}$, for some explicit function $f$ depending on the parameters $A>0$ and $w \in \C$. 
This concentration result is stated as a first moment bound restricted to some barrier event, see Lemma~\ref{lem:concentration_W_infty}, and is proved via first and second moment estimates.

To prove the almost sure convergence in Theorem~\ref{thm:as}, we use this concentration result with $A=1$, combined with other concentration results for $Z_{t}(f)$ around $\rho(f) Z_\infty$ deduced from moment estimates in \cite{MaiPai2019,MaiPai2021} to obtain
\begin{equation} \label{eq:pointwise_bound}
\E \left[ \1_{G_K} \left| \frac{W_\infty(z)}{1-z} - 2 Z_\infty \right| \right]
\leq C(K) (1-\beta)^{1/4},
\end{equation}
for $z = \beta +i \tau \in \cP_1$ and some event $G_K$ which becomes very likely as the parameter $K$ becomes large, see Corollary~\ref{cor:concentration_for_as_cv}.
Then, we decompose the cone $\{ z : \abs{\tau} \leq \theta(1-\beta) \}$ into pieces $A_n = \{ z : \abs{\tau} \leq \theta(1-\beta), \beta \in [2^{-n-1},2^{-n}] \}$ and show that 
\begin{equation*} 
\E \left[ \1_{G_K} \sup_{z\in A_n} \left| \frac{W_\infty(z)}{1-z} - 2 Z_\infty \right| \right]
\leq C(K) 2^{-n/4}.
\end{equation*}
For this, the key idea is to use Cauchy integral formula to bound the supremum by an integral on a contour around $A_n$ but contained in $\cP_1$ and then switch the integral and the expectation to apply \eqref{eq:pointwise_bound}. 
Then, the almost sure convergence follows from Borel--Cantelli lemma.
The idea of bounding the supremum in terms of an integral comes from the proof of the uniform convergence of additive martingales by Biggins \cite{Big1992}.

To prove Theorem~\ref{thm:fluctu}, we first use Lemma~\ref{lem:concentration_W_infty} to show that the error made when replacing $W_\infty(z)/(1-z)$ by the quantity of the type $Z_{At}(f)$ is negligible as $t \to \infty$ and then $A \to \infty$. Then, we can directly apply results from \cite{MaiPai2021} to get the fluctuations of this quantity around $2Z_\infty$ for any fixed $A$. Finally, we take the limit $A \to \infty$ to conclude the proof.

The paper is organized as follows. In Section~\ref{sec:preliminary}, we first recall the many-to-few lemmas and then turn to the proof of the aforementioned concentration results, including Lemma~\ref{lem:concentration_W_infty}.
Theorem~\ref{thm:as} and Proposition~\ref{prop:inprob} are proved in Section~\ref{sec:as}.
Section~\ref{sec:fluctu} contains the proofs of Theorem~\ref{thm:fluctu} and Corollary~\ref{cor:fluctu}.

\section{Preliminary results}
\label{sec:preliminary}

\subsection{Many-to-few lemmas}

In this section, we work with the BBM started from arbitrary position $x\in\R$ and denote the corresponding probability and expectation by $\P_x$ and $\E_x$. 

We first recall the classical many-to-one lemma, see e.g.\ \cite[Section 4.1]{MaiPai2019}.

\begin{lem}[Many-to-one]
	For any $x\in\R$, $t\ge 0$ and any measurable non-negative function $F \colon \cC([0,t])\to \R_+$ where $\cC([0,t])$ is the set of all real-valued continuous functions on $[0,t]$ equipped with uniform norm, we have
	\begin{equation}\label{many-to-1}
	\E_x\left[ \sum_{u\in\calN_t} e^{-X_u(t)} F(X_u(s), s\in[0,t]) \right] = e^{-x} \E_x\left[ F(B_s, s\in[0,t])\right],
	\end{equation}
	where $(B_s)_{s\ge0}$ is a standard 1-dimensional Brownian motion started from $x$ under $\P_x$.
\end{lem}

In particular, for any bounded and measurable function $\varphi\colon\R\to\C$, the many-to-one lemma shows that
\begin{equation*}
\E_x\left[ \sum_{u\in\calN_t} e^{-X_u(t)} \varphi(X_u(t)) \ind{\forall s\in [0,t], X_u(s) >0} \right] = e^{-x} \E_x\left[ \varphi(B_t) \ind{\forall s\in[0,t], B_s >0} \right].
\end{equation*}
Let $(R_s)_{s\ge 0}$ be a 3-dimensional Bessel process, starting from $x$ under $\P_x$. Then, by \cite{Imh1984},
one can rewrite the previous equality as
\begin{equation}\label{M-t-1}
\E_x\left[ \sum_{u\in\calN_t} e^{-X_u(t)} \varphi(X_u(t)) \ind{\forall s\in [0,t], X_u(s) >0} \right] 
= e^{-x} \E_x\left[\frac{x}{R_t}\varphi(R_t) \right].
\end{equation}
Note that, for $x,t>0$, the law of $R_t$ under $\P_x$ is 
\begin{equation}\label{densityRt}
\P_x( R_t \in \diff y ) = \ind{y>0}\frac{y}{x} \left( e^{-\frac{(y-x)^2}{2t}} - e^{-\frac{(y+x)^2}{2t}} \right) \frac{\diff y }{\sqrt{2\pi t}},
\end{equation}
whereas the law of $R_1$ under $\P_0$ is the measure $\rho$ introduced in \eqref{eq:cv_Z_t(f)}.

Moreover, under the assumption that $\E[L^2]<\infty$, the many-to-two formula holds as well. 
We use the following specific form, which can be found in \cite[Eq.\ (3.3)]{MaiPai2021}: for any $x\in\R$, $t\ge0$ and any measurable bounded function $\varphi \colon \R\to\C$, we have 
\begin{multline}\label{many-to-2}
\E_x\left[ \abs{ \sum_{u\in\calN_t} e^{-X_u(t)} \varphi(X_u(t) ) \ind{\forall s\in[0,t], X_u(s)>0} }^2\right] \\
=\E[L(L-1)] e^{-x} \int_0^t \diff r \int_0^\infty \Biggl\lvert \E_y\Biggl[ \sum_{u\in\calN_{t-r}}e^{-X_u(t-r)} \varphi(X_u(t-r)) \ind{\forall s\in [0,t-r], X_u(s) > 0 } \Biggr]^2 e^y q_r(x,y) \diff y\\
+e^{-x} \int_0^\infty e^{-y} \abs{\varphi(y)}^2 q_t(x,y) \diff y,
\end{multline}
where $q_r(x,y) \coloneqq (2\pi r)^{-1/2} \bigl( e^{-(x-y)^2/(2r)} - e^{-(x+y)^2/(2r)} \bigr)$ for $x,y,r>0$ is the transition kernel of the Brownian motion killed at 0. 
Note that the equation in \cite{MaiPai2021} is stated for $\varphi$ with nonnegative values, but this version follows by polarization identity.

\subsection{Moment estimates for the BBM killed at 0}

For any $t_0\geq 0$ and $\gamma \in \R$, we introduce a modification of the additive martingale where particles are killed when going below level $\gamma$ after time $t_0$: for any $z \in \C$ and $s \geq t_0$, we set
\begin{equation} \label{eq:def_Wtilde}
\widetilde{W}_{s}^{t_0,\gamma}(z)
\coloneqq \sum_{u\in\calN_s} e^{-zX_u(s)-(1-z)^2s/2} \ind{\forall r \in [t_0,s], X_u(r) > \gamma}.
\end{equation}
In this section, we establish a first moment estimate and a second moment bound for $\widetilde{W}_{s}^{0,0}(z)$.

The following lemma deal with the first moment of $\widetilde{W}_{s}^{0,0}(z)$. 
It first gives the limit as $s \to \infty$ and then provides an upper bound for any $s \geq 0$. This bound is not intended to be sharp; rather, it is simplified for convenient use in the proof of the second moment bound in Lemma~\ref{lem:second_moment_0}.

\begin{lem} \label{lem:first_moment_0}
	For any $z\in\C$, $x>0$ and $s>0$,
	\begin{equation} \label{eq:first_moment_0}
	\E_x \left[ \widetilde{W}_{s}^{0,0}(z) \right]
	= e^{-x} G(x,z,s),
	\end{equation}
	where we set
	\begin{equation} \label{eq:def_G}
	G(x,z,s) \coloneqq e^{-(1-z)^2s/2} 
	\int_0^\infty e^{(1-z) y} 
	\left( e^{-(y-x)^2/(2s)} - e^{-(y+x)^2/(2s)} \right)
	\frac{\diff y}{\sqrt{2\pi s}}.
	\end{equation}
	Moreover, for any $x>0$ and $z = \beta + i \tau$ with $\beta \in [0,1)$ and $\abs{\tau} \leq 1-\beta$, we have
	\begin{equation} \label{eq:CV_G}
	G(x,z,s) \xrightarrow[s\to\infty]{} 2 \sinh((1-z)x),
	\end{equation} 
	and, for any $s\geq 0$, 
	\begin{equation} \label{eq:bound_G}  
	\abs{G(x,z,s)}
	\leq 4 (x+1) e^{2(1-\beta)x} \left((1-\beta) 
	+ \left( \frac{1}{(1-\beta)\sqrt{s}} \wedge 1 \right)  \right).
	\end{equation}
\end{lem}

\begin{proof}
	Let $z = \beta + i \tau$ with $\beta,\tau\in\R$, $x >0$ and $s>0$.
	By the many-to-one lemma in the form of \eqref{M-t-1}, we get
	\begin{equation} \label{eq:first_moment_step1}
	\E_x \left[ \widetilde{W}_{s}^{0,0}(z) \right]
	= x e^{-x-(1-z)^2s/2} 
	\E_x \left[ \frac{e^{(1-z) R_s}}{R_s} \right].
	\end{equation}
	Using the density of $R_s$ under $\P_x$ given in \eqref{densityRt}, we obtain
	\begin{align*}
	\E_x \left[ \frac{e^{(1-z) R_s}}{R_s} \right]
	&= \int_0^\infty \frac{e^{(1-z) y}}{y} \cdot \frac{y}{x}
	\left( e^{-(y-x)^2/(2s)} - e^{-(y+x)^2/(2s)} \right)
	\frac{\diff y}{\sqrt{2\pi s}}.
	\end{align*}
	This, combined with \eqref{eq:first_moment_step1}, proves \eqref{eq:first_moment_0}.
	
	In order to control the function $G$, we first compare it with the following function:
	\begin{equation} \label{eq:def_G_0}
	G_0(x,z,s) \coloneqq e^{-(1-z)^2 s/2} 
	\int_{\mathbb{R}} e^{(1-z) y} 
	\left( e^{-(y-x)^2/(2s)} - e^{-(y+x)^2/(2s)} \right)
	\frac{\diff y}{\sqrt{2\pi s}},
	\end{equation}
	which is defined similarly, except that the domain of integration has been extended to the full real line.
	Note that by the formula for the moment generating function of the Gaussian distribution, we have
	\begin{equation} \label{eq:G_0}
	G_0(x,z,s) = e^{(1-z)x} - e^{-(1-z)x} 
	= 2 \sinh((1-z)x), \qquad \forall s>0.
	\end{equation}
	We now control the difference between $G$ and $G_0$. Noting that $e^{-(y-x)^2/(2s)} \leq e^{-(y+x)^2/(2s)}$ for $x>0$ and $y<0$, we obtain 
	\begin{align*}
	|G(x,z,s)-G_0(x,z,s)|
	&\leq e^{-[(1-\beta)^2-\tau^2]s/2} 
	\int_{-\infty}^0 e^{(1-\beta) y} 
	e^{-(y+x)^2/(2s)} \frac{\diff y}{\sqrt{2\pi s}} \\
	&= e^{-(1-\beta)x+\tau^2 s/2} 
	\int_{-\infty}^0 e^{-(y+x-(1-\beta)s)^2/(2s)} 
	\frac{\diff y}{\sqrt{2\pi s}}.
	\end{align*}
	If $x < (1-\beta)s/2$, we use the Gaussian tail bound $\int_a^\infty e^{-t^2/2} \frac{\diff t}{\sqrt{2\pi}} \leq \bigl(\frac{1}{a} \wedge 1\bigr) e^{-a^2/2}$ for $a>0$ to get 
	\begin{align}
	|G(x,z,s)-G_0(x,z,s)|
	&\leq e^{-(1-\beta)x+\tau^2 s/2} 
	\left( \frac{\sqrt{s}}{(1-\beta)s-x} \wedge 1 \right)
	e^{-((1-\beta)s-x)^2/(2s)} \nonumber \\
	&\leq e^{-x^2/(2s)} e^{[\tau^2-(1-\beta)^2] s/2} 
	\left( \frac{2}{(1-\beta)\sqrt{s}} \wedge 1 \right) \nonumber \\
	&\leq \frac{2}{(1-\beta)\sqrt{s}} \wedge 1, \label{eq:bound_G_first_case}
	\end{align}
	using the fact that $|\tau| \leq 1-\beta$.
	This expression vanishes as $s \to \infty$ for fixed $x$ and $z$; therefore, together with \eqref{eq:G_0}, this proves \eqref{eq:CV_G}.
	Moreover, using that the exponential function is $e^a$-Lipschitz on $\{ w \in \C : \re w \leq a\}$, we also have the bound $|G_0(x,z,s)| \leq 2e^{(1-\beta)x}|1-z|x$ and thus obtain
	\begin{align*}
	|G(x,z,s)|
	&\leq 2e^{(1-\beta)x}|1-z|x + 
	\left( \frac{2}{(1-\beta)\sqrt{s}} \wedge 1 \right),
	\end{align*}
	which proves \eqref{eq:bound_G} (note that $|1-z| \leq \sqrt{2}(1-\beta)$) in the case $x < (1-\beta)s/2$.
	
	If $x \ge (1-\beta)s/2$, then we estimate $G$ directly from its definition in \eqref{eq:def_G}:
	\begin{align*} 
	|G(x,z,s)| 
	&\le e^{-[(1-\beta)^2-\tau^2]s/2} 
	\int_0^\infty e^{(1-\beta) y} 
	\left(e^{-(y-x)^2/(2s)} - e^{-(y+x)^2/(2s)}\right)
	\frac{\diff y}{\sqrt{2\pi s}} \\
	&= e^{(1-\beta)x+\tau^2 s/2} 
	\int_0^\infty e^{-(y-(1-\beta) s-x)^2/(2s)}
	\left(1-e^{-2xy/s}\right)
	\frac{\diff y}{\sqrt{2\pi s}}.
	\end{align*}
	On the one hand, bounding $1-e^{-2xy/s} \le 1$ yields $|G(x,z,s)| \le e^{(1-\beta)x+\tau^2 s/2}$.
	On the other hand, using the bound $1-e^{-2xy/s} \le 2xy/s$, we obtain
	\begin{align*}  
	|G(x,z,s)|
	&\le e^{(1-\beta)x+\tau^2 s/2} \cdot \frac{2x}{s} \int_{\mathbb{R}} |y| \frac{e^{-(y-(1-\beta) s-x)^2/(2s)}}{\sqrt{2\pi s}} \diff y \\
	&\le e^{(1-\beta)x+\tau^2 s/2} \cdot \frac{2x}{s} \bigl((1-\beta)s + x + \sqrt{s}\bigr).
	\end{align*}
	Therefore,
	\begin{align*}  
	|G(x,z,s)|
	&\le 2 e^{(1-\beta)x+\tau^2 s/2} \cdot \left[ 1 \wedge \left((1-\beta)x +\frac{x^2}{s}+\frac{x}{\sqrt{s}} \right) \right] \\
	&\le 2 e^{(1-\beta)x+\tau^2 s/2} \cdot \left[ (1-\beta)x + \left( \frac{2x}{\sqrt{s}} \wedge 1 \right) \right],
	\end{align*}
	where we used that if $x/\sqrt{s} \le 1$ then $x^2/s \le x/\sqrt{s}$, and we bounded the minimum by $1$ in the opposite case.
	Finally, using $|\tau| \le 1-\beta$ and $x \ge (1-\beta)s/2$, we have $e^{\tau^2 s/2} \le e^{(1-\beta)^2 s/2} \le e^{(1-\beta)x}$, which proves \eqref{eq:bound_G} in this case, thereby completing the proof.
\end{proof}

We bound the second moment of $\widetilde{W}_{s}^{0,0}(z)$ in the next lemma.
\begin{lem} \label{lem:second_moment_0}
	There exists a constant $C_{\eqref{eq:2mom_0}}>0$ such that for any $z = \beta + i \tau$ with $\beta \in [\tfrac45,1)$ and $|\tau| \le 1-\beta$, any $x>0$ and any $s\ge 1$,
	\begin{equation} \label{eq:2mom_0}
	\E_x \left[ \bigl| \widetilde{W}_{s}^{0,0}(z) \bigr|^2 \right]
	\le C_{\eqref{eq:2mom_0}} \, e^{-x}
	\Bigl( 1 + \frac{x}{\sqrt{s}} \Bigr)
	\Bigl( (1-\beta)^2 + \frac{\log(s+1)}{(1-\beta)^2 s} \Bigr).
	\end{equation}
\end{lem}

\begin{proof}
	By the many‑to‑two formula \eqref{many-to-2}, we have
	\begin{align}
	\E_x \left[ \bigl| \widetilde{W}_{s}^{0,0}(z) \bigr|^2 \right]
	&= \E[L(L-1)] e^{-x} \int_0^s \left( \int_0^\infty \bigl| e^{-(1-z)^2 r/2} \,
	\E_y \bigl[ \widetilde{W}_{s-r}^{0,0}(z) \bigr] \bigr|^2
	e^y \, q_r(x,y) \,\diff y \right) \diff r \nonumber \\
	&\quad + e^{-x} \int_0^\infty \bigl| e^{-z y - (1-z)^2 s/2} \bigr|^2 e^{y} \, q_s(x,y) \,\diff y, \label{eq:second_moment_step1}
	\end{align}
	where $q_r(x,y) \coloneqq (2\pi r)^{-1/2} \bigl( e^{-(x-y)^2/(2r)} - e^{-(x+y)^2/(2r)} \bigr)$ for $x,y,r>0$.
	Note that $q_r(x,\cdot)$ is the density at time $r$ of a Brownian motion starting from $x$ and killed at $0$.
	
	\underline{Estimate of the second term.}
	Using the bound $q_s(x,y) \le C_1 x y s^{-3/2}$ for all $x,y,s>0$, the second term in \eqref{eq:second_moment_step1} is at most
	\begin{equation} \label{eq:second_moment_part1}
	\frac{C_1 x e^{-x}}{s^{3/2}} \, e^{-[(1-\beta)^2-\tau^2]s}
	\int_0^\infty y e^{(1-2\beta)y} \,\diff y
	\le \frac{C_{\eqref{eq:second_moment_part1}} x e^{-x}}{s^{3/2}},
	\end{equation}
	where we used $\beta \ge 4/5$ to bound the integral and $|\tau|\le 1-\beta$ to control the exponential factor.
	
	\underline{Estimate of the first term.}
	We now focus on the first term in \eqref{eq:second_moment_step1}.
	By Lemma \ref{lem:first_moment_0},
	\begin{align*}
	\bigl| \E_y \bigl[ \widetilde{W}_{s-r}^{0,0}(z) \bigr] \bigr|
	&\le 4 (y+1) e^{(1-2\beta)y}
	\Bigl( (1-\beta) + \Bigl( \frac{1}{(1-\beta)\sqrt{s-r}} \wedge 1 \Bigr) \Bigr).
	\end{align*}
	Since $|e^{-(1-z)^2 r/2}| \le 1$, the first term is bounded by
	\begin{equation*}
	C_2 e^{-x} \int_0^s \Bigl( (1-\beta)^2 + \frac{1}{(1-\beta)^2(s-r+1)} \Bigr)
	\Bigl( \int_0^\infty (y+1)^2 e^{(3-4\beta)y} \, q_r(x,y) \,\diff y \Bigr) \diff r.
	\end{equation*}
	We split the $r$‑integral into $\int_0^{s/2}$ and $\int_{s/2}^s$.
	
	\textit{The part $r\in[0,s/2]$.}
	For $r\le s/2$ we have $s-r+1 \ge s/2$, hence
	\begin{align}
	& C_2 e^{-x} \Bigl( (1-\beta)^2 + \frac{2}{(1-\beta)^2 s} \Bigr)
	\int_0^\infty (1+y)^2 e^{(3-4\beta)y}
	\Bigl( \int_0^{s/2} q_r(x,y) \,\diff r \Bigr) \diff y \nonumber \\
	&\le C_{\eqref{eq:second_moment_part2}} e^{-x} \Bigl( (1-\beta)^2 + \frac{1}{(1-\beta)^2 s} \Bigr). \label{eq:second_moment_part2}
	\end{align}
	The last inequality uses $\int_0^\infty q_r(x,y) \,\diff r = 2(x\wedge y) \le 2y$ (the potential kernel of Brownian motion killed at $0$) together with $\beta \ge 4/5$.
	
	\textit{The part $r\in[s/2,s]$.}
	Again using $q_r(x,y) \le C_1 x y r^{-3/2}$, we obtain
	$\int_0^\infty (y+1)^2 e^{(3-4\beta)y} q_r(x,y) \,\diff y \le C_3 x r^{-3/2}$.
	Consequently,
	\begin{equation} \label{eq:second_moment_part3}
	C_3 x e^{-x} \int_{s/2}^s \Bigl( (1-\beta)^2 + \frac{1}{(1-\beta)^2(s-r+1)} \Bigr) \frac{\diff r}{r^{3/2}}
	\le C_{\eqref{eq:second_moment_part3}} x e^{-x}
	\Bigl( \frac{(1-\beta)^2}{\sqrt{s}} + \frac{\log(s+1)}{(1-\beta)^2 s^{3/2}} \Bigr).
	\end{equation}
	
	Combining the estimates \eqref{eq:second_moment_part1}, \eqref{eq:second_moment_part2} and \eqref{eq:second_moment_part3} yields the desired bound \eqref{eq:2mom_0}.
\end{proof}

\subsection{Comparing martingale limits with finite-time quantities}

Instead of directly comparing $W_\infty(z)/(1-z)$ with $2Z_\infty$, we introduce suitable finite‑time quantities as an intermediate step.  
For any measurable function $f \colon \mathbb{R} \to \C$, any time $t \ge 0$, and any shift parameter $\gamma \in \mathbb{R}$, define
\begin{equation} \label{eq:def_Z_t(f,gamma)}
Z_t(f, \gamma) 
\coloneqq \sum_{u \in \mathcal{N}_t} (X_u(t) - \gamma)_+ \, e^{-X_u(t)} f\!\left( \frac{X_u(t) - \gamma}{\sqrt{t}} \right).
\end{equation}
We are particularly interested in the family of test functions indexed by $w \in \mathbb{C}$:
\begin{equation} \label{eq:def_F_w}
F_w(x) 
\coloneqq 2 e^{-w^2/2} \cdot \frac{\sinh(wx)}{wx}, \qquad x \in \mathbb{R}.
\end{equation}
For any interval $I \subset [0,\infty)$ and $\gamma \in \mathbb{R}$, we also introduce the event on which all particles stay above the level $\gamma$ over $I$:
\begin{equation} \label{eq:def_E}
E_{I,\gamma} \coloneqq \bigl\{ \forall s \in I,\; \forall u \in \mathcal{N}_s,\; X_u(s) > \gamma \bigr\}.
\end{equation}
Our first goal is to compare $W_\infty(z)/(1-z)$ with $Z_{At}(F_{\sqrt{A}w},\gamma)$ on suitable good events.

\begin{lem} \label{lem:concentration_W_infty}
	Let $\delta \in (0,1)$. There exists a constant $C_{\eqref{eq:concentration_W_infty}}>0$ (depending only on $\delta$) such that for any $t \geq 5\delta^{-2}$, $A>0$, $K>0$, $|\gamma| \leq \sqrt{t}$, and $z = 1 - \frac{w}{\sqrt{t}}$ with $w \in \{ a+ib : a\in [\delta,\delta^{-1}],\; |b| \leq a \}$, we have
	\begin{equation}\label{eq:concentration_W_infty}
	\E \left[ \1_{E_{[0,At],-K} \cap E_{[At,\infty),\gamma}} \left| \frac{W_\infty(z)}{1-z} - e^{(1-z)\gamma} Z_{At}(F_{\sqrt{A}w},\gamma) \right| \right]
	\leq \frac{C_{\eqref{eq:concentration_W_infty}} K^{1/2}}{t^{1/4} A^{1/4}} e^{-\gamma/2}.
	\end{equation}
\end{lem}
\begin{proof}
	Write $z = \beta + i \tau$. Under our assumptions on $t$ and $w$, we have $\beta \in [\tfrac45,1)$ and $|\tau| \leq 1-\beta$. Let $s \geq 1$.
	Recall the definition of $\widetilde{W}_{At+s}^{At,\gamma}(z)$ in \eqref{eq:def_Wtilde}. By the triangle inequality,
	\begin{align}
	& \E \left[ \1_{E_{[0,At],-K} \cap E_{[At,\infty),\gamma}} \left| W_\infty(z) - (1-z)e^{(1-z)\gamma} Z_{At}(F_{\sqrt{A}w},\gamma) \right| \right] 
	\nonumber \\
	& \leq \E \left[ \1_{E_{[At,\infty),\gamma}} \left| W_\infty(z) - \widetilde{W}_{At+s}^{At,\gamma}(z) \right| \right] 
	\nonumber \\
	& \quad + \E \left[ \1_{E_{[0,At],-K}} \left| \widetilde{W}_{At+s}^{At,\gamma}(z) - \E \bigl[ \widetilde{W}_{At+s}^{At,\gamma}(z) \mid \mathcal{F}_{At} \bigr] \right| \right] 
	\nonumber \\
	& \quad + \E \left[ \left| \E \bigl[ \widetilde{W}_{At+s}^{At,\gamma}(z) \mid \mathcal{F}_{At} \bigr] - (1-z)e^{(1-z)\gamma} Z_{At}(F_{\sqrt{A}w},\gamma) \right| \right] 
	\label{eq:decompo_expec}
	\end{align}
	We will bound each of the three expectations on the right-hand side of \eqref{eq:decompo_expec} and then take the limit as $s\to\infty$ in each bound.
	
	\underline{Third term.}
	By the branching property at time $At$,
	\begin{align}
	\E \bigl[ \widetilde{W}_{At+s}^{At,\gamma}(z) \mid \mathcal{F}_{At} \bigr]
	& = \sum_{u\in\mathcal{N}_{At}} e^{-(1-z)^2 At/2}
	\E_{X_u(At)} \bigl[ \widetilde{W}_s^{0,\gamma}(z) \bigr] \nonumber \\
	& = \sum_{u\in\mathcal{N}_{At}} e^{-w^2 A/2 - z\gamma}
	\E_{X_u(At)-\gamma} \bigl[ \widetilde{W}_s^{0,0}(\beta_t) \bigr] \nonumber \\
	& = e^{(1-z)\gamma}
	\sum_{u\in\mathcal{N}_{At}} e^{-X_u(At)-w^2 A/2}
	G\bigl( X_u(At)-\gamma, z, s \bigr) \1_{\{X_u(At)>\gamma\}},
	\label{eq:1st_moment_given_At}
	\end{align}
	by Lemma \ref{lem:first_moment_0}. Using \eqref{eq:CV_G} to send $s\to\infty$,
	\begin{align*}
	\E \bigl[ \widetilde{W}_{At+s}^{At,\gamma}(z) \mid \mathcal{F}_{At} \bigr]
	\xrightarrow[s\to\infty]{\text{a.s.}}
	& e^{(1-z)\gamma}
	\sum_{u\in\mathcal{N}_{At}} e^{-X_u(At)-w^2 A/2}
	2\sinh\bigl( (1-z)(X_u(At)-\gamma) \bigr) \1_{\{X_u(At)>\gamma\}} \\
	& = (1-z)e^{(1-z)\gamma} Z_{At}(F_{\sqrt{A}w},\gamma),
	\end{align*}
	by the definitions \eqref{eq:def_Z_t(f,gamma)} and \eqref{eq:def_F_w}.
	From \eqref{eq:bound_G} we have $|G(x,z,s)| \leq \frac{8}{1-\beta}(x+1)e^{2(1-\beta)x}$, which provides a dominating function for the right-hand side of \eqref{eq:1st_moment_given_At}. Hence, by the dominated convergence theorem,
	\begin{equation}
	\limsup_{s\to\infty} \E \left[ \left| \E \bigl[ \widetilde{W}_{At+s}^{At,\gamma}(z) \mid \mathcal{F}_{At} \bigr] - (1-z)e^{(1-z)\gamma} Z_{At}(F_{\sqrt{A}w},\gamma) \right| \right] = 0
	\label{eq:concentration_W_3rd_term}.
	\end{equation}
	
	\underline{Second term.}
	Applying the conditional Cauchy–Schwarz inequality given $\mathcal{F}_{At}$,
	\begin{align}
	\E \left[ \1_{E_{[0,At],-K}} \left| \widetilde{W}_{At+s}^{At,\gamma}(z) - \E \bigl[ \widetilde{W}_{At+s}^{At,\gamma}(z) \mid \mathcal{F}_{At} \bigr] \right| \right] 
	\leq \E \left[ \1_{E_{[0,At],-K}} \operatorname{Var}\bigl( \widetilde{W}_{At+s}^{At,\gamma}(z) \mid \mathcal{F}_{At} \bigr)^{1/2} \right]. \label{eq:bound_2nd_term}
	\end{align}
	Using the branching property at time $At$ again,
	\begin{align}
	& \operatorname{Var}\bigl( \widetilde{W}_{At+s}^{At,\gamma}(z) \mid \mathcal{F}_{At} \bigr) \nonumber \\
	& = \sum_{u\in\mathcal{N}_{At}} \bigl| e^{-(1-z)^2 At/2 - z\gamma} \bigr|^2
	\operatorname{Var}_{X_u(At)-\gamma}\bigl( \widetilde{W}_s^{0,0}(z) \bigr) \nonumber \\
	& \leq C_{\eqref{eq:bound_variance}} e^{-\operatorname{Re}(w^2)A - (2\beta-1)\gamma} 
	\left( \frac{1}{t} + \frac{t}{s} \log(s+1) \right)
	\sum_{u\in\mathcal{N}_{At}} e^{-X_u(At)} 
	\left( 1 + \frac{X_u(At)-\gamma}{\sqrt{s}} \right) \1_{\{X_u(At)>\gamma\}}, \label{eq:bound_variance}
	\end{align}
	by Lemma \ref{lem:second_moment_0} and the fact that $(1-\beta)^2 = (\operatorname{Re} w)^2/t \in [\delta^2/t, \delta^{-2}/t]$.
	For any $s \geq At$, the sum on the right-hand side of \eqref{eq:bound_variance} is bounded by $Z_{At}(f,\gamma) / \sqrt{At}$ with $f(x) = (x^{-1}+1)\1_{x>0}$. Bounding $e^{-\operatorname{Re}(w^2)A} \leq 1$ and returning to \eqref{eq:bound_2nd_term}, we obtain
	\begin{align}
	\limsup_{s\to\infty} \E \left[ \1_{E_{[0,At],-K}} \left| \widetilde{W}_{At+s}^{At,\gamma}(z) - \E \bigl[ \widetilde{W}_{At+s}^{At,\gamma}(z) \mid \mathcal{F}_{At} \bigr] \right| \right] 
	& \leq \frac{C_{\eqref{eq:bound_variance}}^{1/2} e^{-(\beta-1/2)\gamma}}{t^{3/4} A^{1/4}} \E \left[ \1_{E_{[0,At],-K}} Z_{At}(f,\gamma)^{1/2} \right] 
	\nonumber \\
	& \leq \frac{C_{\eqref{eq:concentration_W_2nd_term}} e^{-\gamma/2 + \gamma \operatorname{Re} w/\sqrt{t}}}{t^{3/4} A^{1/4}} K^{1/2}
	\label{eq:concentration_W_2nd_term},
	\end{align}
	where we used $\E[\1_{E_{[0,At],-K}} Z_{At}(f,\gamma)^{1/2}] \leq \E[\1_{E_{[0,At],-K}} Z_{At}(f,\gamma)]^{1/2} \leq C_4 K^{1/2}$ by applying \cite[Eq.~(3.16)]{MaiPai2021}. Since $\gamma \operatorname{Re} w/\sqrt{t} \leq \delta^{-1}$ by our assumptions, the desired bound follows.
	
	\underline{First term.}
	On the event $E_{[At,\infty),\gamma}$, we have $\widetilde{W}_{At+s}^{At,\gamma}(z) = W_{At+s}(z)$. The point $z$ lies in the region where the additive martingale is uniformly integrable (recall $\beta \in [4/5,1)$ and $|\tau| \leq \beta$)\footnote{do we have $L^1$ convergence on the blue lines? Yes!}, so $W_{At+s}(z)$ converges in $L^1$ to $W_\infty(z)$ as $s\to\infty$. Consequently,
	\begin{equation}
	\limsup_{s\to\infty} \E \left[ \1_{E_{[At,\infty),\gamma}} \left| W_\infty(z) - \widetilde{W}_{At+s}^{At,\gamma}(z) \right| \right] = 0
	\label{eq:concentration_W_1st_term}.
	\end{equation}
	
	Returning to the decomposition \eqref{eq:decompo_expec}, using the bounds \eqref{eq:concentration_W_3rd_term}, \eqref{eq:concentration_W_2nd_term} and \eqref{eq:concentration_W_1st_term}, and multiplying by $|1-z|^{-1} \leq \delta^{-1}\sqrt{t}$, we obtain the desired inequality.
\end{proof}

In order to estimate $Z_t(F_w,\gamma)$, we first compute $\rho(F_w)$, recalling the definition of $\rho$ in \eqref{eq:cv_Z_t(f)}.

\begin{lem} \label{lem:rho(F_w)}
	For any $w\in \C$, $\rho(F_w) = 2$. 
\end{lem}

\begin{proof} For any $w\in\C$ and $x\in\R_+$, 
	\begin{align*}
	\rho(F_w) 
 = \int_0^\infty 2 e^{-w^2/2} \frac{\sinh(wx)}{wx} 
	\cdot \sqrt{\frac{2}{\pi}} x^2 e^{-x^2/2} \diff x 
	& = 2 e^{-w^2/2} \int_0^\infty \sum_{k=0}^\infty \frac{(wx)^{2k}}{(2k+1)!} \cdot \sqrt{\frac{2}{\pi}} x^2 e^{-x^2/2} \diff x.
	\end{align*}
	Fubini--Lebesgue theorem implies that
	\begin{align*}
	\rho(F_w) 
	& =2 e^{-w^2/2} \sum_{k=0}^\infty  \frac{w^{2k}}{(2k+1)!}  \sqrt{\frac{2}{\pi}}\int_0^\infty x^{2k+2}e^{-x^2/2} \diff x \\
	& = 2 e^{-w^2/2} \sum_{k=0}^\infty \frac{w^{2k}}{(2k+1)!} \sqrt{\frac{2}{\pi}} \int_0^\infty (2r)^{k+\frac12} e^{-r } \diff r \quad \textrm{by change of variables } r = x^2/2\\
	& = 2 e^{-w^2/2} \sum_{k=0}^\infty \frac{w^{2k}}{(2k+1)!} \sqrt{\frac{1}{\pi}}  2^{k+1}\Gamma \left(k+ \frac32 \right),
	\end{align*}
	where $\Gamma$ is the usual Gamma function. As $\Gamma(k+\frac32) = (k+\frac12)(k-\frac12)\cdots \frac12\Gamma(\frac12) =
	\frac{(2k+1)!}{k! \cdot 2^{2k+1}} \sqrt{\pi}$,
	we obtain that
	\begin{align*}
	\rho(F_w) 
	&= 2 e^{-w^2/2} \sum_{k=0}^\infty \frac{(w^2/2)^k }{k!} = 2,
	\end{align*}
	which concludes the proof.
\end{proof}

We now establish a concentration estimate for $Z_t(F_w,\gamma)$, comparing it with the shifted derivative martingale $Z_t(1,\gamma)$ up to the multiplicative factor $\rho(F_w) = 2$.

\begin{lem} \label{lem:concentration_Z_t(F_w)}
	Let $\mathcal{K} \subset \mathbb{C}$ be a compact set. There exists a constant $C_{\eqref{eq:concentration_Z_t(F_w)}}>0$ (depending only on $\mathcal{K}$) such that for any $t \geq 1$, $K>0$, $\gamma \in \mathbb{R}$, and $w \in \mathcal{K}$,
	\begin{equation} \label{eq:concentration_Z_t(F_w)}
	\E \left[ \1_{E_{[0,\sqrt{t}],-K} \cap E_{[\sqrt{t},\infty),\gamma}} \bigl| Z_t(F_w,\gamma) - 2Z_t(1,\gamma) \bigr| \right]
	\leq \frac{C_{\eqref{eq:concentration_Z_t(F_w)}}}{t^{1/4}} \left( K + K^{1/2} e^{-\gamma/2} \right).
	\end{equation}
\end{lem}

\begin{proof}
	This follows mainly from \cite[Lemma~4.1]{MaiPai2021}. In their notation, set
	\[
	\widetilde{Z}_t^{\sqrt{t},\gamma}(F_w,\gamma)
	\coloneqq \sum_{u\in\mathcal{N}_t} (X_u(t)-\gamma)_+ e^{-X_u(t)} 
	F_w\!\left( \frac{X_u(t) - \gamma}{\sqrt{t}} \right)
	\1_{\{\forall r \in [\sqrt{t},t],\; X_u(r) > \gamma\}}.
	\]
	Then, on the event $E_{[\sqrt{t},\infty),\gamma}$, we have $Z_t(F_w,\gamma) - 2Z_t(1,\gamma) = \widetilde{Z}_t^{\sqrt{t},\gamma}(F_w-\rho(F_w),\gamma)$, where we will see later in Lemma~\ref{lem:rho(F_w)} that $\rho(F_w)=2$. Consequently, the left-hand side of \eqref{eq:concentration_Z_t(F_w)} is bounded by
	\begin{align}
	&\E \left[ \1_{E_{[0,\sqrt{t}],-K}} \bigl| \widetilde{Z}_t^{\sqrt{t},\gamma}(F_w-\rho(F_w),\gamma) \bigr| \right] \nonumber \\
	&\leq \E \left[ \1_{E_{[0,\sqrt{t}],-K}} \left| \E \bigl[ \widetilde{Z}_t^{\sqrt{t},\gamma}(F_w-\rho(F_w),\gamma) \mid \mathcal{F}_{\sqrt{t}} \bigr] \right| \right] \nonumber \\
	&\quad + \E \left[ \1_{E_{[0,\sqrt{t}],-K}} \operatorname{Var}\bigl( \widetilde{Z}_t^{\sqrt{t},\gamma}(F_w-\rho(F_w),\gamma) \mid \mathcal{F}_{\sqrt{t}} \bigr)^{1/2} \right] \nonumber \\
	&\leq \frac{C_\eqref{eq:bound_Z_t(F_w)}}{t^{1/4}} \E \left[ \1_{E_{[0,\sqrt{t}],-K}} Z_{\sqrt{t}} \bigl( x \mapsto x^2 e^{\kappa x}+1,\; \gamma \bigr) \right] \nonumber \\
	&\quad + \frac{C_\eqref{eq:bound_Z_t(F_w)} e^{-\gamma/2}}{t^{1/4}} \E \left[ \1_{E_{[0,\sqrt{t}],-K}} Z_{\sqrt{t}} \bigl( x \mapsto x^{-1}+1,\; \gamma \bigr)^{1/2} \right],
	\label{eq:bound_Z_t(F_w)}
	\end{align}
	where we applied \cite[Lemma~4.1]{MaiPai2021} with $\kappa>0$ chosen large enough (depending on $\mathcal{K}$) so that every $F_w$ ($w\in\mathcal{K}$) satisfies the assumptions (A1$_\kappa$)–(A2$_\kappa$) therein.
	Finally, by \cite[Eq.~(3.16)]{MaiPai2021}, the two expectations on the right-hand side of \eqref{eq:bound_Z_t(F_w)} are bounded by $C_5K$ and $C_5K^{1/2}$, respectively. This concludes the proof.
\end{proof}

Finally, we compare $Z_t(1,\gamma)$ with $Z_\infty$ in the following lemma.

\begin{lem} \label{lem:concentration_Z_infty}
	There exists a constant $C_{\eqref{eq:concentration_Z_infty}}>0$ such that for any $t \ge 1$, $K>0$, and $\gamma \in \mathbb{R}$,
	\begin{equation}\label{eq:concentration_Z_infty}
	\E \left[ \1_{E_{[0,t],-K} \cap E_{[t,\infty),\gamma}} \bigl| Z_\infty - Z_t(1,\gamma) \bigr| \right]
	\le \frac{C_{\eqref{eq:concentration_Z_infty}} K^{1/2}}{t^{1/4}} e^{-\gamma/2}.
	\end{equation}
\end{lem}

\begin{proof}
	This follows from arguments in \cite{MaiPai2019}. For $s \ge t$, set
	\[
	\widetilde{Z}_s^{t,\gamma}
	\coloneqq \sum_{u\in\mathcal{N}_s} (X_u(s)-\gamma)_+ e^{-X_u(s)} \1_{\{\forall r \in [t,s],\; X_u(r) > \gamma\}}.
	\]
	By \cite[Theorem~9]{Kyp2004}, $(\widetilde{Z}_s^{t,\gamma})_{s\ge t}$ is a nonnegative martingale whose limit is denoted by $\widetilde{Z}_\infty^{t,\gamma}$.
	Note that $Z_t(1,\gamma) = \widetilde{Z}_t^{t,\gamma}$ and, on the event $E_{[t,\infty),\gamma}$, we have $Z_\infty = \widetilde{Z}_\infty^{t,\gamma}$.
	Therefore, for any $s \ge t$,
	\begin{equation} \label{eq:decompo_Z}
	\E \left[ \1_{E_{[0,t],-K} \cap E_{[t,\infty),\gamma}} \bigl| Z_\infty - Z_t(1,\gamma) \bigr| \right]
	\le \E \left[ \bigl| \widetilde{Z}_\infty^{t,\gamma} - \widetilde{Z}_s^{t,\gamma} \bigr| \right]
	+ \E \left[ \1_{E_{[0,t],-K}} \bigl| \widetilde{Z}_s^{t,\gamma} - \widetilde{Z}_t^{t,\gamma} \bigr| \right].
	\end{equation}
	By \cite[Theorem~13(ii)]{Kyp2004}, $(\widetilde{Z}_s^{t,\gamma})_{s\ge t}$ is uniformly integrable, so the first term on the right‑hand side of \eqref{eq:decompo_Z} vanishes as $s\to\infty$.
	Using the martingale property, the second term equals
	\begin{equation*}
	\E \left[ \1_{E_{[0,t],-K}} \bigl| \widetilde{Z}_s^{t,\gamma} - \E \bigl[ \widetilde{Z}_s^{t,\gamma} \mid \mathcal{F}_t \bigr] \bigr| \right]
	\le \E \left[ \1_{E_{[0,t],-K}} \operatorname{Var}\bigl( \widetilde{Z}_s^{t,\gamma} \mid \mathcal{F}_t \bigr)^{1/2} \right]
	\le \E \left[ \1_{E_{[0,t],-K}} \bigl( C_5 e^{-\gamma} W_t(1) \bigr)^{1/2} \right],
	\end{equation*}
	by \cite[Eq.\@ (5.4)]{MaiPai2019}%
	\footnote{In that reference, the bound is stated for a truncated version of the martingale because the reproduction law may have infinite second moment. Under our finite second moment assumption, the same bound holds without truncation; the conclusion of \cite[Lemma~4.4]{MaiPai2019} becomes $\E[(\widetilde{Z}_s^{t,\gamma})^2] \le C e^{x}$ with the same proof. Alternatively, one could rely on \cite[Lemma~3.4]{MaiPai2021} and proceed as in \eqref{eq:bound_variance}.}
	By \cite[Eq.~(3.16)]{MaiPai2021}, we have $\E\bigl[ \1_{E_{[0,t],-K}} W_t(1)^{1/2} \bigr] \le C_4 K^{1/2}/t^{1/4}$. Letting $s\to\infty$ in \eqref{eq:decompo_Z} yields the desired inequality.
\end{proof}

\section{Almost sure convergence: proof of Theorem~\ref{thm:as}}
\label{sec:as}

This section is devoted to the proof of Theorem~\ref{thm:as}, which states that almost surely, $W_\infty(z)/(1-z)$ converges to $2Z_\infty$ as $z \to 1$ non-tangentially in $\cP_1$.
To this end, we introduce a family of good events indexed by a parameter $K>0$:
\begin{equation} \label{eq:def_G_K}
G_K \coloneqq E_{[0,K],-K} \cap E_{[K,\infty),0},
\end{equation}
where the events $E_{I,\gamma}$ are defined in \eqref{eq:def_E}.
It is well known that $\min_{u\in\mathcal{N}_t} X_u(t) \to \infty$ almost surely. This was first stated in \cite[Eq.\@ (20)]{LalSel1987} with a somewhat unconvincing argument, but it follows directly from the fact, proved by \cite{Nev1988}, that $W_t(1)$ converges to $0$ almost surely.
Consequently,
\begin{equation} \label{eq:proba_G_K}
\P(G_K) \xrightarrow[K\to\infty]{} 1.
\end{equation}
We first establish a concentration result for $W_\infty(z)/(1-z)$ around its limit $2Z_\infty$ as $z\to 1$, in terms of a first-moment bound on the good event $G_K$. This is a direct consequence of the lemmas established in Section~\ref{sec:preliminary}.

\begin{cor} \label{cor:concentration_for_as_cv}
	There exists a constant $C_{\eqref{eq:concentration_for_as_cv}}>0$ such that for any $K\geq 1$, $t \geq t_K:=100 \vee K^2$ and $z = 1 - \frac{w}{\sqrt{t}}$ with $w \in \{ a+ib : a\in [1/4,4],\; |b| \leq a \}$, we have
	\begin{equation} \label{eq:concentration_for_as_cv}
	\E \left[ \1_{G_K} \left| \frac{W_\infty(z)}{1-z} - 2 Z_\infty \right| \right]
	\leq \frac{C_{\eqref{eq:concentration_for_as_cv}} K}{t^{1/4}}.
	\end{equation}
\end{cor}

\begin{proof}
	By the triangle inequality, the left-hand side of \eqref{eq:concentration_for_as_cv} is at most
	\begin{align*} 
	\E \left[ \1_{G_K} \abs{\frac{W_\infty(z)}{1-z} - Z_t(F_w,0)} \right]
	+ \E \left[ \1_{G_K} \abs{Z_t(F_w,0) - 2 Z_t(1,0)} \right] 
	+ 2 \cdot \E \left[ \1_{G_K} \abs{Z_t(1,0) - Z_\infty} \right].
	\end{align*}
	The result follows by applying, with $\gamma = 0$, Lemma~\ref{lem:concentration_W_infty} to the 1st term, Lemma~\ref{lem:concentration_Z_t(F_w)} to the 2nd term and Lemma~\ref{lem:concentration_Z_infty} to the 3rd term.
\end{proof}

As a direct consequence, we obtain Proposition~\ref{prop:inprob}.

\begin{proof}[Proof of Proposition~\ref{prop:inprob}]
	The result follows from Markov's inequality, Corollary~\ref{cor:concentration_for_as_cv} and \eqref{eq:proba_G_K}.
\end{proof}

Now we are ready to prove the almost sure convergence in Theorem~\ref{thm:as}.

\begin{proof}[Proof of Theorem~\ref{thm:as}]
	Fix $\theta\in(0,1)$ and recall we aim at proving the almost sure convergence of $\frac{W_\infty(z)}{1-z}$ with $z = \beta+i\tau$ as $\beta \nearrow 1$ uniformly in $\abs{\tau} \leq \theta (1-\beta)$. For any $t \geq 100$, we introduce the set
	\[
	A_t \coloneqq \left\{ 1-\frac{a+i b}{\sqrt{t}} : a \in \left[1, 2\right],\; \abs{b} \le \theta a \right\}.
	\]
	The idea is to cover the cone $\{ z=\beta+i\tau : \beta < 1, \abs{\tau} \leq \theta (1-\beta) \}$ in the neighborhood of 1 by the union of the sets $A_{2^k}$ for all $k$ large enough.
	We also consider a counterclockwise oriented contour $\cC_t$ surrounding $A_t$ inside $\cP_1$: more precisely, let $\theta' = \frac{1+\theta}{2}$ and let $\cC_t$ be the quadrilateral connecting the four points  
	\[
	1-\frac{3(1+i\theta')}{\sqrt{t}},\; 
	1-\frac{3(1-i\theta')}{\sqrt{t}},\; 
	1-\frac{(1-i\theta')}{2\sqrt{t}},\; 
	1-\frac{(1+i\theta')}{2\sqrt{t}},
	\]  
	then $\cC_t$ is indeed included in $\cP_1$ for $t \geq 100$.
	See Figure~\ref{fig:as} for an illustration.
	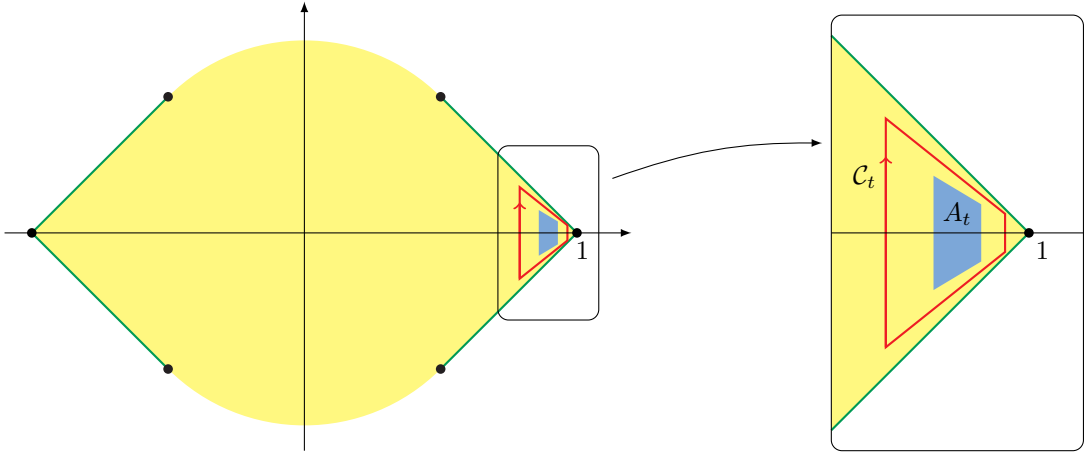
\begin{figure}[tbp]
		\centering
		\begin{tikzpicture}[scale=3, >=latex]
		\def\rad{0.707106781}  
		\def\theta{0.6}
		\def\thetaprime{0.8}
		\def\lw{0.71} 
		\def\rw{1.08} 
		\def\marg{.03} 
		\def\scale{0.07} 
		
		\begin{scope}[scale=1.2]
		\fill[Yellow, fill opacity=0.5, draw=none]
		(1,0) -- (0.5,0.5) 
		arc[start angle=45, end angle=135, radius=\rad] 
		-- (-0.5,0.5) -- (-1,0) 
		-- (-0.5,-0.5) 
		arc[start angle=-135, end angle=-45, radius=\rad] 
		-- (0.5,-0.5) -- cycle;
		
		\draw[ForestGreen, thick] 
		(1,0) -- (0.5,0.5)   
		(0.5,-0.5) -- (1,0)  
		(-1,0) -- (-0.5,0.5) 
		(-0.5,-0.5) -- (-1,0); 
		
		\fill (1,0) circle (0.5pt);
		\fill (-1,0) circle (0.5pt);
		\fill[Black] (0.5,0.5) circle (0.5pt);
		\fill[Black] (0.5,-0.5) circle (0.5pt);
		\fill[Black] (-0.5,0.5) circle (0.5pt);
		\fill[Black] (-0.5,-0.5) circle (0.5pt);
		
		\node at (1.02,0) [below] {$1$};
		
		\fill[RoyalBlue!50, draw=none] (1-2*\scale,2*\theta*\scale) -- (1-\scale,\theta*\scale) -- (1-\scale,-\theta*\scale) -- (1-2*\scale,-2*\theta*\scale) -- cycle;
		
		\draw[Red,thick] (1-3*\scale,3*\thetaprime*\scale) -- (1-\scale/2,\thetaprime*\scale/2) -- (1-\scale/2,-\thetaprime*\scale/2) -- (1-3*\scale,-3*\thetaprime*\scale) -- cycle;
		\draw[Red,thick,->,>=to] (1-3*\scale,-3*\thetaprime*\scale) --  (1-3*\scale,2*\thetaprime*\scale);
		
		\draw[->] (-1.1,0) -- (1.2,0); 
		\draw[->] (0,-0.8) -- (0,0.85); 
		
		\draw[rounded corners] (\lw,1-\lw+\marg) -- (\lw,\lw-1-\marg) -- (\rw,\lw-1-\marg) -- (\rw,1-\lw+\marg) -- cycle;
		
		\draw[->] (1.13,.2) to[bend left=10] (1.9,.33);
		\end{scope}
		\begin{scope}[scale=3, xshift=1.8]
		\fill[Yellow, fill opacity=0.5, draw=none]
		(1,0) -- (\lw,1-\lw) -- (\lw,\lw-1) -- cycle;
		\draw[ForestGreen,thick] (\lw,1-\lw) -- (1,0) -- (\lw,\lw-1);
		
		\fill (1,0) circle (0.2pt);
		\node at (1.02,0) [below] {$1$};
		
		\fill[RoyalBlue!50, draw=none] (1-2*\scale,2*\theta*\scale) -- (1-\scale,\theta*\scale) -- (1-\scale,-\theta*\scale) -- (1-2*\scale,-2*\theta*\scale) -- cycle;
		\node at (1-1.5*\scale,.4*\scale) {$A_t$};
		
		\draw[Red,thick] (1-3*\scale,3*\thetaprime*\scale) -- (1-\scale/2,\thetaprime*\scale/2) -- (1-\scale/2,-\thetaprime*\scale/2) -- (1-3*\scale,-3*\thetaprime*\scale) -- cycle;
		\draw[Red,thick,->,>=to] (1-3*\scale,-3*\thetaprime*\scale) --  (1-3*\scale,2*\thetaprime*\scale);
		\node[below left] at (1-3*\scale,2*\thetaprime*\scale) {$\cC_t$};
		
		\draw (\lw,0) -- (\rw,0);
		\draw[rounded corners] (\lw,1-\lw+\marg) -- (\lw,\lw-1-\marg) -- (\rw,\lw-1-\marg) -- (\rw,1-\lw+\marg) -- cycle;
		\end{scope}
		\end{tikzpicture}
		\caption{Representation of the region $A_t$ in blue and of the contour $\cC_t$ in red.}
		\label{fig:as}
	\end{figure}
	Recall that the map $z \mapsto W_\infty(z)$ is analytic in $\cP_1$ by \cite{Big1992}; hence so is $z \mapsto \frac{W_\infty(z)}{1-z} - 2Z_\infty$. By Cauchy's integral formula, for every $z \in A_t$,  
	\[
	\abs{\frac{W_\infty(z)}{1-z} - 2Z_\infty}
	= \frac{1}{2\pi} \abs{ \int_{\cC_t} \left( \frac{W_\infty(\zeta)}{1-\zeta} - 2Z_\infty \right) \frac{\diff \zeta}{\zeta - z}}.
	\] 
	Next, observe that there exists a constant $c_\theta>0$ such that for any $z\in A_t$ and $\zeta \in \cC_t$, $\abs{\zeta -z} \geq c_\theta/\sqrt{t}$.
	Hence, writing $m(\diff \zeta)$ for the Lebesgue measure on $\cC_t$,
	\[
	\sup_{z\in A_t} \abs{\frac{W_\infty(z)}{1-z} - 2Z_\infty}
	\le \frac{ \sqrt{t} }{c_\theta} 
	\int_{\cC_t} \abs{\frac{W_\infty(\zeta)}{1-\zeta} - 2Z_\infty} m(\diff \zeta).
	\]
	Taking the expectation, applying Fubini theorem and then Corollary~\ref{cor:concentration_for_as_cv} yields that for all $t\ge t_K$,
	\begin{align}\label{eq:uniform_concentration}
	\E\left[ \1_{G_K} \sup_{z\in A_t} \abs{\frac{W_\infty(z)}{1-z} - 2Z_\infty} \right] 
	&\le \frac{ \sqrt{t}}{c_\theta} \int_{\cC_t} \E\left[ \1_{G_K} \abs{\frac{W_\infty(z)}{1-z} - 2Z_\infty} \right] m(\diff \zeta) \nonumber\\
	&\le \frac{\sqrt{t}}{c_\theta}\cdot\frac{ C_{\eqref{eq:concentration_for_as_cv}} K}{ t^{1/4}} \int_{\cC_t} m(\diff \zeta) \nonumber\\
	&\le C_{\eqref{eq:uniform_concentration}} K t^{-1/4},
	\end{align}
	Now set $t = 2^k$ with $k\in\mathbb{N}$. 
	By Markov's inequality and \eqref{eq:uniform_concentration}, for any $K\ge 1$,
	\begin{align*}
	\sum_{k > \log_2 t_K} \P\left( \sup_{z\in A_{2^k}} \abs{\frac{W_\infty(z)}{1-z} - 2Z_\infty} \ge \frac1k,\; G_K \right)
	& \le \sum_{k > \log_2 t_K} k \;\E\left[ \1_{G_K} \sup_{z\in A_{2^k}} \abs{\frac{W_\infty(z)}{1-z} - 2Z_\infty} \right] \\
	& \le C_{\eqref{eq:uniform_concentration}} K \sum_{k\ge 1} k 2^{-k/4} 
	< \infty.
	\end{align*}
	It then follows from the Borel–Cantelli lemma that, on the event $G_K$, almost surely, with $z = \beta+i\tau$,
	\[
	\sup_{\abs{\tau} \leq \theta (1-\beta)} \abs{\frac{W_\infty(z)}{1-z} - 2Z_\infty} \xrightarrow[\beta\nearrow1]{} 0.
	\]
	Since $\P(G_K)\to 1$ as $K\uparrow\infty$ by \eqref{eq:proba_G_K}, this completes the proof.
\end{proof}

\section{Fluctuations: proof of Theorem~\ref{thm:fluctu}}
\label{sec:fluctu}

In this section, we turn to the study of fluctuations in the convergence of $W_\infty(z)/(1-z)$ towards $2Z_\infty$. 
The main idea is to first compare $W_\infty(z)/(1-z)$ with $Z_{At}(F_{\sqrt{A}w},\gamma)$, and then apply the known results on the $1$-stable fluctuations of $Z_{At}(F_{\sqrt{A}w},\gamma)$ around its limit $2Z_\infty$, established in \cite{MaiPai2021}.
The first step is done in the following corollary, obtained as a consequence of Lemma~\ref{lem:concentration_W_infty}.

\begin{cor} \label{cor:concentration_for_fluctu}
	For any $w \in \{ a+ib : a\in (0,\infty),\; |b| \leq a \}$, the following convergence holds, with $\gamma = \frac{1}{2} \log t$ and $z = 1 - \frac{w}{\sqrt{t}}$:
	\begin{equation*} 
	\sqrt{t} \left( \frac{W_\infty(z)}{1-z} - e^{(1-z)\gamma} Z_{At}(F_{\sqrt{A}w},\gamma) \right)
	\xrightarrow[t,A\to\infty]{(\P)} 0,
	\end{equation*}
	where we first let $t\to\infty$ and then $A\to\infty$.
\end{cor}

\begin{proof}
	Fix $w \in \{ a+ib : a\in (0,\infty),\; |b| \leq a \}$ and $\varepsilon>0$. Recall the notation $E_{I,\gamma}$ from \eqref{eq:def_E}. For any $K,A,t>0$, we have
	\begin{align}
	& \P \left( \left| \frac{W_\infty(z)}{1-z} - e^{(1-z)\gamma} Z_{At}(F_{\sqrt{A}w},\gamma) \right|
	\geq \frac{\varepsilon}{\sqrt{t}} \right) 
	\nonumber \\
	& \leq \frac{\sqrt{t}}{\varepsilon} \E \left[ \1_{E_{[0,At],-K} \cap E_{[At,\infty),\gamma}} \left| \frac{W_\infty(z)}{1-z} - e^{(1-z)\gamma} Z_{At}(F_{\sqrt{A}w},\gamma) \right| \right]
	\nonumber \\
	& \quad + \P \left( E_{[0,At],-K}^c \right) 
	+ \P \left( E_{[0,At],-K} \cap E_{[At,\infty),\gamma}^c \right).
	\end{align}
	We now show that each term on the right-hand side vanishes as we successively let $t\to\infty$, then $A\to\infty$, and finally $K\to\infty$.
	\begin{itemize}
		\item By Lemma~\ref{lem:concentration_W_infty}, the first term is bounded by $C_6 K^{1/2} \varepsilon / A^{1/4}$, where $C_6$ depends on $w$ (indeed, we may choose $\delta>0$ according to $w$ so that the constant $C_{\eqref{eq:concentration_W_infty}}$ depends on $w$). Hence this term vanishes as desired.
		\item By \cite[Eq.\@ (D.1)]{MaiPai2019}, the second term is at most $e^{-K}$, which also vanishes.
		\item For the third term, assume that from time $At$ onward, particles are killed when they reach a level less than or equal to $\gamma$. Let $N_{(At,\infty),\gamma}$ denote the total number of particles killed during $(At,\infty)$. Then
		\begin{equation}
		\P \left( E_{[0,At],-K} \cap E_{[At,\infty),\gamma}^c \right)
		\leq \P \left( \min_{u\in\mathcal{N}_{At}} X_u(At) \leq \gamma \right)
		+ \E \left[ \1_{E_{[0,At],-K}} N_{(At,\infty),\gamma} \right].
		\end{equation}
		The first term on the right-hand side vanishes as $t\to\infty$ by \cite[Proposition~3]{Bra1978}.
		For the second term, using \cite[Eq.\@ (3.33)]{MaiPai2021}, we obtain
		\begin{equation}
		\E \left[ \1_{E_{[0,At],-K}} N_{(At,\infty),\gamma} \right]
		\leq e^{\gamma} \E \left[ \1_{E_{[0,At],-K}} W_{At}(1) \right]
		\leq \sqrt{t} \cdot \frac{C_4 K}{\sqrt{At}},
		\end{equation}
		where the last inequality follows from \cite[Eq.~(3.16)]{MaiPai2021}. This bound vanishes as required.
	\end{itemize}
	Thus all terms vanish in the prescribed limit, concluding the proof.
\end{proof}

The proof of Theorem~\ref{thm:fluctu} then mainly follows from results \cite{MaiPai2021}, which we first recall here.
Consider $F\colon \R \to \R$ which is twice differentiable on $(0,\infty)$ and satisfies, for some constant $C>0$, $\abs{F''(x)} \leq C e^{Cx}$ for all $x>0$.
Recalling the definition of $\rho$ in \eqref{eq:cv_Z_t(f)}, define
\begin{equation} \label{eq:def_R}
\mathscr{R}F(r) \coloneqq 
\rho(F(\sqrt{1-r} \cdot{} ))\1_{\{r< 1\}} - \rho(F),
\qquad r \geq 0.
\end{equation}
Fix $a>0$.
Then it follows from \cite[Theorem~1.2]{MaiPai2021} that $\mathscr R F(\cdot/a) \in \cG$ (introduced in \eqref{eq:cG}) and 
\begin{equation} \label{eq:theorem1.2}
\sqrt t \cdot \left( 
Z_{at}(F) - \rho(F) Z_\infty
+ \frac{\log t}{2 \sqrt{t}} 
\int_0^\infty \mathscr R F\left( \frac{r}{a} \right) \frac{Z_\infty}{\sqrt{2\pi}}  \frac{\diff r}{r^{3/2}}
\right)
\xrightarrow[t\to\infty]{(\textup{d})}
\int_0^\infty \mathscr R F \left( \frac{r}{a} \right) M_{Z_\infty}(\diff r),
\end{equation}	
where $M_{Z_\infty}$ was introduced around \eqref{eq:characteristic_function_M}, and the convergence holds jointly for any finite number of test functions satisfying the same assumption as $F$.
Furthermore, letting $\gamma = \frac{1}{2} \log t$, it follows from \cite[Proposition~6.4]{MaiPai2021} (the function $F$ satisfies the assumptions of this result with $\alpha = 0$ and $\kappa=C+1$ up to dividing it by a constant) that
\begin{equation} \label{eq:proposition6.4}
\sqrt{t} \cdot \left(
Z_{at}(F) - Z_{at} (F,\gamma) 
+ \frac{\log t}{2 \sqrt{at}}  
\int_0^\infty \mathscr R F(r) \frac{Z_\infty}{\sqrt{2\pi}} \frac{\diff r}{r^{3/2}}
\right)
\xrightarrow[t\to\infty]{(\P)} 0.
\end{equation}	
Changing $r$ into $r/a$ in the integral in \eqref{eq:proposition6.4} and then combining it with \eqref{eq:theorem1.2}, we get
\begin{equation} \label{eq:combined}
\sqrt t \cdot \left( 
Z_{at}(F,\gamma) - \rho(F) Z_\infty
\right)
\xrightarrow[t\to\infty]{(\textup{d})}
\int_0^\infty \mathscr R F \left( \frac{r}{a} \right) M_{Z_\infty}(\diff r),
\end{equation}	
where the convergence holds jointly for any finite number of test functions satisfying the same assumption as~$F$.
The fact that shifting $Z_{at}(F)$ by $\gamma = \frac{1}{2} \log t$ makes the $(\log t)/\sqrt{t}$ correction term disappear had already been observed in \cite{MaiPai2019,MaiPai2021}.

\begin{proof}[Proof of Theorem~\ref{thm:fluctu}] 
	Recall $\cW = \{ w = a+ib \in \C : a > 0, \abs{b} \leq a \}$.
	We write $\gamma = \frac{1}{2} \log t$ and $z = 1- \frac{w}{\sqrt{t}}$.
	By Corollary~\ref{cor:concentration_for_fluctu}, it is enough to prove that, for any $w\in \cW$ and $A>0$,
	\begin{equation} \label{eq:goal1}
	\sqrt{t} \cdot \left( (e^{(1-z)\gamma}-1) Z_{At}(F_{\sqrt{A}w},\gamma) + \frac{w \log t}{\sqrt{t}} Z_\infty \right) 
	\xrightarrow[t \to \infty]{(\P)} 0.
	\end{equation}
	and that, in the sense of finite-dimensional distributions,
	\begin{equation} \label{eq:goal2}
	\left( \sqrt{t} \cdot \left( Z_{At}(F_{\sqrt{A}w},\gamma) - 2 Z_\infty \right) \right)_{w \in \cW} 
	\xrightarrow[t,A \to \infty]{(\textup{d})} 
	\left( \int_0^\infty 2 \left( e^{-r w^2/2} -1 \right) M_{Z_\infty}(\diff r) \right)_{w \in \cW},
	\end{equation}
	where we first let $t\to\infty$ and then $A\to\infty$.
	
	We start with \eqref{eq:goal2}. 
	Let $w \in \cW$. 
	Recall that $F_w(x) = 2 e^{-w^2/2} G(wx)$ for any $x \in \R$ where we introduced $G(z) = \sinh(z)/z$ for $z \in \C$.
	By a series expansion, observe that $\abs{G''(z)} \leq \sum_{k\geq0} \frac{\abs{z}^{2k}}{(2k+1)!} \leq e^{\abs{z}}$, so that the functions $\re F_w$ and $\im F_w$ both satisfy the assumption under which \eqref{eq:combined} holds.
	Also recall from Lemma~\ref{lem:rho(F_w)} that $\rho(F_w) = 2$.
	Thus, we get, for any $A>0$, in the sense of finite-dimensional distributions, 
	\begin{equation} \label{eq:goal2_partial}
	\left( \sqrt{t} \cdot \left( Z_{At}(F_{\sqrt{A}w},\gamma) - 2 Z_\infty \right) \right)_{w \in \cW} 
	\xrightarrow[t \to \infty]{(\textup{d})} 
	\left( \int_0^\infty \mathscr R F_{\sqrt{A}w} \left( \frac{r}{A} \right) M_{Z_\infty}(\diff r) \right)_{w \in \cW},
	\end{equation}
	Then, noting that for $w \in \C$ and $r \in [0,1)$,  $F_w(\sqrt{1-r} \cdot{}) = e^{-r w^2/2} F_{w \sqrt{1-r}}$ and using again Lemma~\ref{lem:rho(F_w)}, one finds that
	\begin{align*}
	\mathscr{R}F_{\sqrt{A} w}\!\left(\frac{r}{A}\right) 
	= \rho\left(F_{\sqrt{A} w} \left(\sqrt{1-\tfrac{r}{A}}\cdot{} \right)\right) \1_{\{r/A < 1\}} - \rho(F_{\sqrt{A} w})
	= 2\left(e^{-r w^2/2}\1_{\{r<A\}} -1\right),
	\end{align*}
	which converges to $g_w(r) \coloneqq 2(e^{-r w^2/2} -1)$ as $A\to\infty$. 
	It can be easily checked that, for any $w\in\mathbb{C}$, $g_w \in \cG$ so that $\int_0^\infty g_w(r) M_{Z_\infty}(\mathrm{d}r)$ is well-defined. It then follows from a chracteristic function calculation and from the dominated convergence theorem that, in the sense of finite-dimensional distributions, 
	\begin{equation} \label{eq:cv_in_A_of_the_limit}
	\left( \int_0^\infty \mathscr R F_{\sqrt{A}w} \left( \frac{r}{A} \right) M_{Z_\infty}(\diff r) \right)_{w \in \cW}
	\xrightarrow[A \to \infty]{(\textup{d})} 
	\left( \int_0^\infty g_w(r) M_{Z_\infty}(\diff r) \right)_{w \in \cW}.
	\end{equation}
	Together with \eqref{eq:goal2_partial}, this proves \eqref{eq:goal2}.
	
	We now turn to \eqref{eq:goal1}. Fix $w\in \cW$ and $A>0$.
	Noting that $e^{(1-z)\gamma}-1 = - \frac{w \log t}{\sqrt{t}} + o(\frac{1}{\sqrt{t}})$ as $t \to \infty$, the left-hand side of \eqref{eq:goal1} equals
	\begin{equation*}
		-w \log t \cdot \left( Z_{At}(F_{\sqrt{A}w},\gamma) - Z_\infty\right) + o(Z_{At}(F_{\sqrt{A}w},\gamma))
		\xrightarrow[t\to\infty]{(\P)} 0,
	\end{equation*}
	as a consequence of \eqref{eq:goal2_partial}.
\end{proof}

\begin{proof}[Proof of Corollary~\ref{cor:fluctu}]
	Applying Theorem~\ref{thm:fluctu} with $\beta = z = 1- \frac{1}{\sqrt{t}}$, 
	\begin{equation*}
	\frac{1}{1-\beta}\left( \frac{W_\infty(\beta)}{1-\beta} - 2 Z_\infty +2(1-\beta)\log(1-\beta) Z_\infty\right)
	\xrightarrow[\beta \nearrow 1]{(\textup{d})} 
	\int_0^\infty g_1(r) M_{Z_\infty}(\diff r),
	\end{equation*}
	 with $g_1(r) = 2(e^{-r/2}-1)$.
	 By \eqref{eq:characteristic_function_M} (or see \cite[Remark 1.4]{MaiPai2021}), conditionally on $Z_\infty$, $\int_0^\infty g_1(r) M_{Z_\infty}(\diff r)$ is of 1-stable distribution $S_1(\sigma_{g_1}Z_\infty, \beta_{g_1}, \mu_{g_1}Z_\infty)$ where
	 \begin{equation*}
	 \sigma_{g_1} \coloneqq \int_0^\infty \abs{g_1(r)} \frac{\sqrt{\pi}}{2\sqrt{2}}\frac{\diff r}{r^{3/2}}, \quad 
	 \beta_{g_1} \coloneqq \frac{\int_0^\infty g_1(r) \frac{\diff r}{r^{3/2}}}{\int_0^\infty \abs{g_1(r)} \frac{\diff r}{r^{3/2}}}, \quad
	 \mu_{g_1} \coloneqq - \int_0^\infty g_1(r) \frac{2}{\pi} (\log \abs{g_1(r)} - \mu_Z)
	 \frac{\sqrt{\pi}}{2\sqrt{2}}\frac{\diff r}{r^{3/2}}.
	 \end{equation*}
	 Integrating by parts $1/r^{3/2}$, one finds $\sigma_{g_1} = \sqrt{\pi} \Gamma(1/2) = \pi$.
	 Noting that $g_1 \leq 0$, one gets $\beta_{g_1} = -1$.
	 Finally, expanding $\log(1-e^{-r/2}) = -\sum_{k\geq1} e^{-kr/2}/k$ and switching the sum and the integral by Fubini--Tonelli theorem, one gets
	 \begin{align*}
	 \mu_{g_1} 
	 & = \frac{2}{\pi} (\log 2 -\mu_Z) \sigma_{g_1} 
	 +  \sum_{k\geq 1} \int_0^\infty (e^{-r/2}-1) \frac{e^{-kr/2}}{k} 
	 \sqrt{\frac{2}{\pi}} \frac{\diff r}{r^{3/2}} \\
	 & = 2 \Biggl( \log 2 -\mu_Z - \sum_{k\ge 1}\frac{\sqrt{k+1}-\sqrt{k}}{k} \Biggr),
	 \end{align*}
	 using $\sigma_{g_1} = \pi$ and integrating by parts $1/r^{3/2}$ as before.
	 In the series, multiplying the numerator and the denominator by $\sqrt{k+1}+\sqrt{k}$ gives the desired form.
\end{proof}

\section*{Acknowledgements}

Michel Pain would like to thank Vincent Vargas for raising the question of the almost sure convergence of $W_\infty(\beta)/(1-\beta)$ and acknowledges support from the MITI interdisciplinary program 80PRIME GEx-MBB and from the ANR project MBAP-P. Xinxin Chen acknowledges support from National Natural Science Foundation of China (Grant No. 12571148) and National Key R\&D Program of China (No. 2022YFA1006500).

\bibliographystyle{alpha}
\bibliography{biblio}

\end{document}